%% file: draft.tex
\documentclass[onefignum]{siamart250211}

\graphicspath{ {./images/} }
\usepackage{subcaption}
\usepackage{url}
\usepackage{verbatim}
\usepackage{braket, amsfonts, amssymb, amsopn, amscd,stmaryrd,xfrac}
\usepackage{pifont}
\usepackage{longtable}
\usepackage{multirow}
\usepackage{todonotes}

\usepackage{tikz}
\usepackage{pgfplots}
\usepackage{enumitem}
\usepackage{calc}
\usepackage{placeins}

\newsiamremark{assumption}{Assumption}
\crefname{assumption}{Assumption}{Assumptions}
\newsiamremark{remark}{Remark}
\crefname{remark}{Remark}{Remarks}

\pgfplotsset{
    width=\linewidth,
    compat=1.18, 
    legend style={
        at={(0,0)},
        anchor=north,
        legend columns=1,
        cells={anchor=west},
        font=\footnotesize,
        rounded corners=2pt,
    },
    every axis/.append style={font=\small},
}

\usetikzlibrary{patterns,shapes.arrows}
\usetikzlibrary{backgrounds}
\usetikzlibrary{fit}

\DeclareMathOperator{\HS}{HS}

\DeclareMathOperator{\dom}{dom}
\DeclareMathOperator{\Cov}{\mathbf{Cov}}

\DeclareMathOperator*{\argmin}{arg\,min}
\DeclareMathOperator{\diag}{diag}
\DeclareMathOperator{\M}{\mathbb{M}}

\DeclareMathOperator{\loc}{loc}
\DeclareMathOperator{\obs}{obs}

\newcommand*\diff{\mathop{}\!\mathrm{d}}

\newcommand*\RR{\mathbb{R}}

\newcommand{\Cs}{C_{\mathrm{s}}}
\newcommand{\sourcedom}{D}
\newcommand{\SLP}{\mathtt{P}}
\newcommand{\SLO}{S}
\newcommand{\DLO}{D}
\newcommand{\COperator}{\mathcal{C}}
\newcommand{\COp}{\COperator}
\newcommand{\Cq}[1][q]{\COperator_{#1}}
\newcommand{\GreenFunction}[1][]{G#1}

\newcommand{\etaaskappa}{\kappa}

\title{Hölder-Logarithmic Stability and Convergence Rates for an Inverse Random Source Problem\thanks{Submitted to the editors January 30th 2026.}
\funding{P. Mickan was funded by the Deutsche Forschungsgemeinschaft (DFG, German Research Foundation) through RTG 2088. T. Hohage acknowledges support by DFG through CRC 1456 (grant 432680300, project C04).}}
\author{Philipp Ronald Mickan\thanks{Institut f\"ur Numerische und Angewandte Mathematik, Universit\"at G\"ottingen, 37083, Germany (\email{p.mickan@math.uni-goettingen.de}).
}
\and Thorsten Hohage\thanks{Institut f\"ur Numerische und Angewandte Mathematik, Universit\"at G\"ottingen, 37083, Germany and Max-Planck-Institut f\"ur Sonnensystemforschung, G\"ottingen, 37077, Germany (\email{hohage@math.uni-goettingen.de}}}

\headers{Stability Estimates for Random Source Problem}{Philipp R. Mickan, and Thorsten Hohage}

\begin{document}
\maketitle

\input{00_abstract}

\input{01_introduction}

\input{02_MainResults}

\input{03_VSC}

\input{04_Proof}

\input{05_NumericalParts}

\input{06_Conclusion}

\end{document}

%% file: 00_abstract.tex
\begin{abstract}
In this paper, we investigate an inverse random source problem involving the recovery of the strength of a random, uncorrelated acoustic source from correlation measurements of the emitted time-harmonic acoustic waves. 
Such problems arise in applications including aeroacoustics and seismic imaging. 
Unlike their deterministic counterparts, inverse random source problems are known to be uniquely solvable in the absence of noise. 
Nevertheless, due to their inherent ill-posedness, regularization is required to stably reconstruct the source strength.

We derive conditional H\"older-logarithmic stability estimates under Sobolev smoothness assumptions by employing complex geometrical optics solutions. 
Moreover, by establishing a variational source condition, we obtain H\"older-logarithmic convergence rates for spectral regularization methods. 
At a fixed frequency, the exponents in the logarithmic stability and convergence estimates grow without bound as the Sobolev regularity of the source increases.
Finally, we present numerical experiments supporting our theoretical findings.

\end{abstract}



%% file: 01_introduction.tex
\section{Introduction}
\label{sec:Introduction}
Inverse source problems consist of determining the right-hand side of a differential equation from (usually incomplete) measurements of its solution.
Sources produced by uncontrolled processes, such as turbulence, are naturally and advantageously modeled as random processes.
Such models are used in numerous applications, including
passive seismic imaging \cite{Artman:06,Bleistein2001}, 
helioseismic holography \cite{Agaltsov_2020, lindsey2000basic, mullerQuantitativePassiveImaging2024},
noise localization in aeroacoustics \cite{Hohage2020,Kaltenbacher2018},  
and ocean tomography \cite{BSS:08}.
Inverse problems for virtual random sources also appear in the context of imaging in random media \cite{garnierPassiveImagingAmbient2016}. 
The waves excited by the unknown source and measured to extract information about it are realizations of a random solution process of a partial differential equation.  
The stochastic characteristics of these solution processes constitute the idealized, noise-free data of the inverse problem. 

Our work is specifically motivated by noise localization in aeroacoustics and helioseismic holography. Here, the idealized noise-free data are 
the covariance operators of solution processes, whereas the available noisy input data are statistical estimators, namely, sample covariances computed from the measured wave fields. The quantity to be recovered in the inverse problem is the pointwise covariance of the source, which is modeled as an uncorrelated process.

In such settings, uniqueness results can be established \cite{Devaney1979,Hohage2020} even for single-frequency boundary data, which is not possible for deterministic inverse source problems. The same references also show
that without any assumptions on the covariance of a centered Gaussian random source process, the inverse problem of recovering
its covariance structure is not unique. 
The aim of the present paper is to extend the aforementioned uniqueness results to conditional stability estimates 
and convergence rates for a random inverse source problem. 

We first discuss some related work. In \cite{PLi2017a,YZhao2019},
stability is established for the one-dimensional multi-frequency inverse random source problem governed by the Helmholtz equation 
using an explicit inversion formula for the Fourier transform of the source strength. 
Their H\"older-logarithmic stability estimate depends on a low-frequency average of boundary data comprising differences in variances and means and exploits the improved stability of the multi-frequency approach presented in \cite{Cheng2016}. 
These works assume that the random sources are real-valued, whereas in the applications motivating our work, the sources must be modeled by 
complex-valued stochastic processes. 

A different type of result was established in \cite{Lassas2008} for waves scattered by a random potential modeled by a Gaussian random process whose covariance operator is a pseudodifferential operator: the authors showed that the intensities of a solution to the differential equation for a single realization of the random process, averaged over all frequencies, uniquely determine the local strength of the potential. 
Similar results for a random source instead of a random potential were established in \cite{JLi2020,PLi2021c}. 
Furthermore, in recent years, considerable progress has been made in establishing uniqueness and stability results for 
different types of time-domain passive imaging problems, see, e.g., \cite{feizmohammadi:25,HELIN2018132,NHZ:20,TLSK:24}.

The work most closely related to the present paper is \cite{li2023stability}, which establishes logarithmic stability estimates in the frequency domain. 
We improve these results in several important ways. Most importantly, the constants in 
the logarithmic stability estimates in \cite{li2023stability} deteriorate as the  
wave number tends to infinity, whereas our bounds yield a H\"older rate if the
wave number grows sufficiently rapidly as the noise level tends to $0$
(see Remark \ref{rem:Hoelder_logarithmic} for details). Moreover, we establish not only a stability estimate, i.e., an estimate of differences between elements in 
the range of the forward operator, but also convergence rates of spectral regularization 
methods applied to noisy data. Furthermore, we reduce the minimal Sobolev regularity index 
$s$ of the source strength from $s\geq 3$ in \cite{li2023stability} to  
$s>0$, and the asymptotic behavior of the exponent in logarithmic stability estimates 
for monochromatic data is improved from $s/3$ to $s$. 
Reference \cite{li2023stability} also contains a H\"older-type bound, but, as detailed at the
end of \Cref{sec:MainResults}, it is not a classical stability estimate; such an estimate is impossible in this setting. 

\medskip
The remainder of this paper is organized as follows. We state the main results in the next section. 
\Cref{sec:VSC} introduces the method used to verify variational source conditions. 
Next, \Cref{sec:ProofMainProposition} introduces complex geometrical optics solutions, which are used to bound the Fourier coefficients of the source strength in terms of the covariance data. These bounds provide the tools needed to prove the main results, \cref{thm:StabilityEstimate,thm:ConvergenceRate}, via variational source conditions. 
In \Cref{sec:Numerical}, we report numerical experiments supporting the convergence rates established in \Cref{thm:ConvergenceRate}. 
Finally, we present our conclusions.

%% file: 02_MainResults.tex
\section{Setting and main results}\label{sec:MainResults}
\paragraph*{Problem formulation}
Most of this paper deals with the random inverse source problem for the Helmholtz equation
\begin{equation}
    \label{eq:Helmholtz}
    \Delta u + \kappa^2u=-Q \text{ in } \mathbb{R}^3
\end{equation}
with a random source $Q$ in a bounded domain $\sourcedom\subset\mathbb{R}^3$. 

We describe $Q$ as a Hilbert space process on $L^2(\sourcedom)$. Recall that, for an underlying probability space $(\Omega,\Sigma,\mathbf{P})$ and a separable Hilbert space $X$, a Hilbert space process $Q$ is a bounded linear mapping from 
$X$ to $L^2(\Omega,\mathbf{P})$. The image of a test function $f\in X$ under $Q$ will be denoted by $\langle Q,f\rangle$, and it is a complex-valued random variable with finite second moments. The expectation $\mathbf{E}[Q]\in X$ and the covariance operator $\mathbf{Cov}[Q]\in L(X)$ are uniquely characterized by 
\[
\langle \mathbf{E}[Q],f\rangle = \operatorname{E}[\langle Q,f\rangle],\qquad 
\langle \mathbf{Cov}[Q]f,g\rangle = \operatorname{Cov}(\langle Q,f\rangle,\langle Q,g\rangle)
\]
for all $f,g\in X$, where $\operatorname{E}$ and $\operatorname{Cov}$ denote the standard expectation and 
covariance of complex-valued random variables. $Q$ is called a \emph{white noise process} if $\mathbf{Cov}[Q]=\sigma^2I$ for some $\sigma>0$, 
and it is called \emph{Gaussian} if $\langle Q,f_1\rangle,\dots,\langle Q,f_n\rangle$ is jointly Gaussian 
for every finite collection $f_1,\dots,f_n\in X$. Given an orthonormal basis $\{f_j:j\in\mathbb{N}\}$ of $X$, 
white noise with $\sigma=1$ can be defined by $\langle Q,f\rangle:=\sum_{j=1}^{\infty} \xi_j\langle f,f_j\rangle$ 
or $Q=\sum_{j=1}^{\infty}\xi_jf_j$ 
with independent scalar complex Gaussian random variables $\xi_j\sim \mathcal{CN}(0,1)$. This definition can 
be shown to be independent of the choice of $\{f_j\}$. White noise is not a random variable in $X$ because its norm 
is infinite almost surely. However, white noise on $L^2(\sourcedom)$ can be shown to be a random variable 
in Sobolev spaces $H^{-s}_0(D)$ for $s>d/2$ (see \cite[Thm. 4.4.3]{GN:15} or \cite[\S 2.2]{HL:25}). For an operator $T\in L(X,Y)$ mapping to another 
separable Hilbert space $Y$, the relation $\langle TQ,g\rangle:=\langle Q,T^*g\rangle$, $g\in Y$, defines a Hilbert space process in $Y$
with 
\begin{align}\label{eq:CovTQ}
\mathbf{Cov}[TQ]=T\mathbf{Cov}[Q]T^*. 
\end{align}
This allows us to define Hilbert space processes with arbitrary bounded positive semidefinite covariance operators by using the square root in the sense of the functional calculus.

We make the following assumptions on the random source \(Q\).
\begin{assumption}[source process]
\label{ass:SourceAssumptions}
Let $\sourcedom$ be an open and bounded domain in $\mathbb{R}^3$ such that $\mathbf{0}\in \sourcedom \subset B(r) \subset \RR^3$. Assume that $Q$ is an \emph{uncorrelated centered Gaussian random process} on $L^2(\sourcedom)$, referred to as the \emph{random source}. That is, we assume that $Q$ has zero mean, $\mathbb{E}[Q]\equiv0\in L^2(\sourcedom)$, and that there exists a function $q\in L^\infty(\sourcedom,[0,\infty))$, called the \emph{source power function}, such that the covariance operator $\Cov[Q]\in\mathcal{L}(L^2(\sourcedom))$ is the multiplication operator $M_q\colon L^2(\sourcedom)\to L^2(\sourcedom)$, 
\begin{equation}
    \label{eq:CovarianceOperator}
    (\Cov[Q]v)(y)=(M_qv)(y)=q(y)v(y)
\end{equation}
for all $y\in\sourcedom$ and $v\in L^2(\sourcedom)$.
\end{assumption}
In the following $\sourcedom$ will be referred to as the source region. 

In contrast to the setting in \cite{Hohage2020}, we assume the measurement region $\M$ to be the boundary of a bounded domain $\mathcal{O}\subset \mathbb{R}^3$:
\begin{assumption}[measurement surface]
  \label{ass:Measurement}
  Let $\M$ be the $C^{2,\alpha}$-smooth boundary of a bounded domain $\mathcal{O}$ that is 
  star-shaped with respect to the origin. Define 
  $R:=\sup\{\Vert x\Vert :x\in\M\} <\infty$, and assume that
  $\inf\{\Vert x\Vert :x\in\M\}>r$
  and $\delta_{\M} := \inf_{x\in \M} (x\cdot \nu)>0$,
  where $\nu$ denotes the exterior unit normal vector on $\M$.
%   \todo[inline]{Die folgende Annahme ist jetzt in der $\inf$-Bedingung enthalten:  
%   $\overline{B(\Tilde{r})}\subset \mathcal{O}$ for some $\Tilde{r}>r$. Entweder das Infimum 
%   $\Tilde{r}$ nennen oder $\Tilde{r}$ einführen, wo es gebraucht wird.}
%   \todo[inline]{This follows from our new assumptions, so we should prove it where needed 
%   instead of assuming it. 
%   Since $\delta_{\M}>0$, we can parameterize $\M$ by a radial function on the unit sphere, 
%   the gradient of this function must be bounded as continuous function on a compact set, 
%   and the surface area can be expressed in terms of this gradient.}
%   Moreover, suppose that $\M$ has a finite surface area. 
  %and  that $\kappa^2$ is not a Dirichlet eigenvalue of the negative Laplacian $-\Delta$ in $\mathcal{O}$.
\end{assumption}

The data from which we seek to recover the random source are given by the sample covariance 
\begin{equation*}
    \Cov_N(x,y) := \frac{1}{N}\sum_{l=1}^N u_l(x)\overline{u_l(y)} \text{ for } x,y \in \M
\end{equation*}
of $N$ realizations $u_l$ of $u|_{\M}$, where $u$ is the random solution process describing time-harmonic sound propagation in a homogeneous background medium and satisfying the Helmholtz equation \eqref{eq:Helmholtz} driven by the random source term $Q$. In the continuous model, this sample covariance approximates the covariance operator of the random process $u|_{\M}$ on $L^2(\M)$ \cite{Hohage2020}.

With the well-known fundamental solution to the Helmholtz equation \eqref{eq:Helmholtz}, 
\begin{equation}
    \label{eq:GreensFunction}
    \GreenFunction[(x,y)]:=\frac{\exp(i\kappa|x-y|)}{4\pi|x-y|}
\end{equation}
(cf.~\cite[Section 2.2.]{ColtonKressIAaES}), 
the solution process of \eqref{eq:Helmholtz} is described by
\begin{equation}
    \label{eq:SolutionAsIntegral}
    u(x)=\int_\sourcedom\GreenFunction[(x,z)]Q(z)\diff z,\qquad x\in\mathbb{R}^3.
\end{equation}
We refer to \cite{HL:25} for regularity conditions on $q$ under which \eqref{eq:SolutionAsIntegral} almost surely defines a distributional solution 
to \eqref{eq:Helmholtz} in $D$. 
By Assumption \ref{ass:Measurement}, the time-harmonic pressure signal can be expressed as the image under the \emph{volume-potential operator} $\mathcal{G}\colon L^2(\sourcedom)\to L^2(\M)$ defined by
\begin{equation}
\label{eq:VolumePotentialOperator}
(\mathcal{G}\psi)(x) := \int_\sourcedom \GreenFunction[(x,z)]\psi(z)\diff z \quad \text{ for } x\in \M.
\end{equation}
By Assumption \ref{ass:Measurement}, the source and measurement regions are separated by a positive distance. Because $\GreenFunction$ is infinitely smooth 
off the diagonal, $u|_{\M}=\mathcal{G}(Q)$ is infinitely smooth.
By \eqref{eq:CovTQ}, the covariance operator of $u|_{\M}$ is given by 
\[
\mathbf{Cov}[u|_{\M}] = \mathcal{G} M_q\mathcal{G}^\ast.
\]
It is not difficult to see (see \cite[Lemma A.6]{HL:25}) that $\mathbf{Cov}[u|_{\M}]$ can be written as an integral operator  
with a square-integrable kernel given by the covariance function
\begin{equation*}
   c_q(x,y):= \mathtt{Cov}(u(x),u(y))= \int_\sourcedom \GreenFunction[(x,z)] q(z) \overline{\GreenFunction[(y,z)]} \diff z 
   \quad  \text{ for } x,y\in\M.
\end{equation*}
Note that the covariance operator $\mathbf{Cov}[u|_{\M}]$ 
only depends on the deterministic source strength $q$ of the random source process $Q$. Thus, the forward operator of our inverse problem  
is deterministic and given by 
\begin{align}
\label{eq:ForwardOp}
\begin{aligned}
 \mathcal{C}\colon L^\infty(\sourcedom,[0,\infty))&\to \HS(L^2(\M)) \\
 q&\mapsto \mathcal{G}M_q\mathcal{G}^\ast .
\end{aligned}
\end{align}
Here $\HS(L^2(\M))$ denotes the space of Hilbert--Schmidt operators in $L^2(\M)$, which is 
the space of compact linear operators $A\colon L^2(\M)\to L^2(\M)$ with square summable 
singular values, equipped with either the $\ell^2$-norm of the sequence of singular values 
or the $L^2(\M\times \M)$-norm of the integral kernel $c_q$:
\[
\Vert \mathcal{C}(q)\Vert_{\HS(L^2(\M))}=\Vert c_q\Vert_{L^2(\M\times \M)}
\] 
(cf. \cite[Theorem VI.23]{ReedSimon1980}). 
This implies that the forward operator $\mathcal{C}$ has a unique continuous extension to a bounded linear 
operator $\mathcal{C}:L^2(\sourcedom)\to \HS(L^2(\M))$ with norm bounded by 
\[
\|\mathcal{C}\|_{L^2(\sourcedom)\to \HS(L^2(\M))} \leq \|(x,y,z)\mapsto
\GreenFunction[(x,z)] \overline{\GreenFunction[(y,z)]}\|_{L^2(\M\times\M\times\sourcedom)}. 
\]
The right-hand side is finite because the kernel is infinitely smooth and $\sourcedom$ and $\M$ are bounded.

\paragraph*{Stability estimates and convergence rates}
A function $\phi\colon[0,\infty) \to [0,\infty)$ is called an index function if it is continuous and monotonically increasing, with $\phi(0) = 0$. 
Let $X$ and $Y$ be Hilbert spaces, let
$F\colon D(F)\subset X\to Y$ be a forward operator, and let $U\subset D(F)$ be a set of candidate solutions satisfying some prior information. An inequality of the form
\begin{equation}
    \label{eq:DefinitionConditionalStability}
    \Vert f_1-f_2\Vert_X\leq \phi\big(\Vert F(f_1)-F(f_2)\Vert_Y\big) \text{ for all } f_1,f_2\in U 
\end{equation}
is called a \emph{conditional stability estimate with index function 
$\phi$ for the set $U$}. Such stability estimates characterize the
modulus of continuity of $(F|_U)^{-1}$ (or generalized inverses) and thus provide insight into the degree of ill-posedness of the inverse problem under the prior information
that the true solution belongs to $U$.

\begin{theorem}
\label{thm:StabilityEstimate}
Suppose Assumptions \ref{ass:SourceAssumptions} and \ref{ass:Measurement} are satisfied. 
Let $m\geq 0$ and $s,\Cs>0$ with $m<s$. Suppose also that $\kappa\geq 1/R$, with $R$ defined in 
Assumption \ref{ass:Measurement}. 
Moreover, let $Y=\HS(L^2(\M))$ and \(X=H^m(\sourcedom)\). Then there exists a constant $C>0$, depending on $m,s,\Cs,\kappa$, and $\M$, such that, for all $q_1,q_2\in \dom(\mathcal{C})$ satisfying $\Vert q_j\Vert_{H^s}\leq \Cs$ for $j=1,2$, a \emph{logarithmic conditional stability estimate} with index function
\begin{align}
\label{eq:ExplicitStabilityEstimate}
\phi(\delta)=C\log\big(3+\delta^{-2}\big)^{-(s-m)}
\end{align}
holds. More precisely, we have a \emph{H\"older-logarithmic conditional stability estimate} with index function
\begin{align}\label{eq:Hoelderlog_stability}
\begin{aligned}
&\phi_{\kappa}(\delta)^2:=
C'\overline{\rho}(\delta,\kappa)^{2+\tau}\!(3+\delta^{-2})^{\frac{1}{3}}\delta + \frac{16}{3}C_s^2 
\overline{\rho}(\delta,\kappa)^{-2(s-m)},\\
&\overline{\rho}(\delta,\kappa):=
\frac{\log(3+\delta^{-2})}{3R}\sqrt{1+\frac{(6R\kappa)^2}{\log(3+\delta^{-2})}},\quad\tau:=\max(2m+\frac{3}{2}-s,0),
\end{aligned}
\end{align}
where $C'$ depends on $m,s$, and $\M$, but not on $\kappa$ 
(see Remark \ref{rem:Hoelder_logarithmic}).
\end{theorem}

Second, we establish convergence rates for spectral regularization methods
\[ q_\alpha = R_\alpha g^{\text{obs}},\;\text{where}\; R_\alpha = r_\alpha(\mathcal{C}^\ast \mathcal{C})\mathcal{C}^\ast,\]
using arguments provided in \cite{hohage2017}. Examples of regularization methods covered by this analysis include 
\emph{Tikhonov regularization}; \emph{iterated Tikhonov regularization}, \emph{modified spectral cut-off}, and 
\emph{Landweber iteration}, see \cite[Assumption 3.2]{hohage2017}. 
Such methods must be complemented by a rule $\alpha = \hat{\alpha}(\delta,\mathcal{C}^{\delta})$ 
determining how to choose the regularization parameter based on the noise level $\delta$ and the noisy data 
$\mathcal{C}^{\delta}$. Roughly speaking, a parameter choice rule is called \emph{weakly quasioptimal} 
if it is asymptotically optimal up to a multiplicative constant (see \cite[Definition 2.1]{RausHaemarik2007}). 
Weak quasioptimality has been demonstrated for many commonly used parameter choice rules including 
the discrepancy principle for methods with infinite qualification and the Lepskii balancing principle (see 
\cite{RausHaemarik2007}). 

\begin{theorem}
\label{thm:ConvergenceRate}
Suppose Assumptions \ref{ass:SourceAssumptions} and \ref{ass:Measurement} are satisfied. Let $m\geq 0$, $s>m$, and $\Cs>0$, and suppose that $\kappa\geq 1/R$, with $R$ defined in 
Assumption \ref{ass:Measurement}. If the exact solution 
$q^\dagger$ satisfies $q^\dagger\in H^s(\sourcedom)$ with $\Vert q^\dagger\Vert_{H^s}\leq \Cs$ and $R_\alpha$ is a spectral regularization 
method satisfying Assumption 3.2 in \cite{hohage2017} and 
$\hat{\alpha}$ is a weakly quasioptimal parameter choice rule, then for any 
$\mathcal{C}^\delta\in \HS(L^2(\M))$ with deterministic noise 
level $\Vert \Cq[q^\dagger]-\mathcal{C}^\delta\Vert_{\HS}<\delta$, the estimator $q_{\hat{\alpha}}^\delta=R_{\hat{\alpha}(\delta,\mathcal{C}^{\delta})}(\mathcal{C}^\delta)$ satisfies the error bound
\begin{align}
\label{eq:ExplicitConvergenceRate}
\Vert q_{\hat{\alpha}}^\delta-q^\dagger\Vert_{H^m}\leq C''\phi_{\kappa}(\delta)
\end{align}
with the function $\phi_{\kappa}$ from \cref{eq:Hoelderlog_stability} and a constant $C''$ depending 
on the regularization method. In particular, there exists a constant $C>0$ depending only on $m,s,\Cs,\kappa$, and $\M$ such that 
$\Vert q_{\hat{\alpha}}^\delta-q^\dagger\Vert_{H^m}\leq C''C\log\big(3+\delta^{-2}\big)^{-(s-m)}$.
\end{theorem}

Examples of weakly quasioptimal parameter choice rules include the
Lepski\u{i} principle for all of the above-mentioned regularization schemes
and the discrepancy principle for regularization methods with infinite classical qualification (cf. \cite{RausHaemarik2007}). 

\begin{remark}[H\"older-logarithmic rates]\label{rem:Hoelder_logarithmic}
First, note that the number $3$ in \cref{eq:ExplicitStabilityEstimate,eq:Hoelderlog_stability} 
could be replaced by any positive number greater than $1$ without changing the 
asymptotic behavior of the functions $\phi$ and $\phi_{\kappa}$ 
at $0$. It has been included to avoid the singularity of $\log(x)^{-1}$ at $x=1$.

In addition, note that, for fixed $\kappa$, the first term
in $\phi_{\kappa}^2$ is negligible in the limit $\delta\to 0$, 
and the second term behaves like $\phi^2$. However, 
for large values of $\kappa$ such that 
$(R\kappa)^2\gg \log(3+\delta^{-2})$, the second term 
becomes negligible. Hence the first 
term dominates and exhibits H\"older-type behavior as $\delta\to 0$, albeit with a large constant. 
Overall, we obtain H\"older-type behavior 
if $(6R\kappa)^2\sim \delta^{-2x}\log(3+\delta^{-2})$  with  $0<x<\frac{1}{3}\frac{1}{\max(2m+7/2-s,2)}$.
In this joint limit of small noise levels $\delta$ and 
large wave numbers $\kappa$, we have 
\begin{equation*}
    \phi_{\kappa}(\delta) \leq C''' \delta^\zeta \text{ with } \zeta := \frac{2(s-m)}{3\max\left(2(1+s-m),\frac{7}{2}+s\right)}
\end{equation*}
with a constant $C'''$ depending on $m,s,\Cs,x$, and $\M$.
\end{remark}

Having formulated our main results, we now discuss 
the H\"older-type estimate in Theorem 1.1 of \cite{li2023stability} mentioned in the introduction. 
The constant $C_1$ in this theorem depends on the wave number, 
but the wave number in turn depends on the data norm (corresponding to 
$\|F(f_1)-F(f_2)\|$ in \eqref{eq:DefinitionConditionalStability}) 
and tends to infinity as the data norm tends to $0$. 
Therefore, \cite[Theorem 1.1]{li2023stability} does not present a classical stability 
estimate in the sense of \eqref{eq:DefinitionConditionalStability}. In fact, as the 
forward operator at fixed frequency is infinitely smoothing, H\"older-type stability estimates 
are in principle impossible. This has been rigorously proved in the PhD thesis \cite[Prop.~4.5.6]{Mickan:25} 
using entropy arguments developed in \cite{dichristo2003,KRS:21,mandache2001}.

%% file: 03_VSC.tex
\section{Variational source conditions}
\label{sec:VSC}
\emph{Variational source conditions} for general nonsmooth operators between Banach spaces, $F\colon X\to Y$, were introduced in \cite{Hofmann_2007}
to deduce convergence rates for Tikhonov regularization. 
These conditions on $f^\dagger\in \dom(F)$ are formulated as a variational inequality 
\begin{equation}
    \label{eq:DefinitionVSC}
\forall f\in\dom(F):\quad 
    \frac{1}{2}\Vert f^\dagger-f\Vert^2_X \leq \Vert f\Vert_X^2 -\Vert f^\dagger\Vert^2_X + \Phi \left(\Vert F(f)-F(f^\dagger)\Vert_Y^2\right)
\end{equation}
with an index function $\Phi$. 
The aim of this section is to review a general method for the verification of such conditions and to show how 
a variational inequality of the form \cref{eq:DefinitionVSC} implies the two main results, \cref{thm:StabilityEstimate,thm:ConvergenceRate}.

We first argue that, if a variational source condition with index
function $\Phi$ is satisfied for all $f^\dagger$ in some subset $U\subset \dom(F)$, then a conditional stability estimate holds for $U$. 
Let $f_1,f_2\in U$. Without loss of generality, assume that $\Vert f_2\Vert \leq \Vert f_1\Vert$. Setting $f^\dagger=f_1$ and $f=f_2$ in \cref{eq:DefinitionVSC} then yields a conditional stability estimate of the form \cref{eq:DefinitionConditionalStability} with the index function $\phi=\sqrt{2\Phi (\cdot^2)}$. 
The results in \cite{hohage2017} establish an equivalence between variational source 
conditions and convergence rates of spectral regularization methods meeting certain assumptions.

Hence, establishing a variational source condition for the random inverse source problem in \Cref{sec:ProofMainProposition} proves both main theorems. 
To find an explicit expression for the index function $\Phi$ that determines the behavior of the conditional stability estimate \cref{eq:DefinitionConditionalStability} and the convergence rate \cref{eq:ExplicitConvergenceRate},
we rely on a characterization of variational source conditions established in \cite{hohage2017}. This method uses a family of subspaces of $X$ and corresponding projections to separate the problem into two sufficient conditions: a generalized smoothness assumption and an ill-posedness estimate. These conditions are as follows.

Let $f^\dagger \in \dom(F)$, let $\mathcal{I}$ be an index set, and assume that there exists a family of projections $P_r\in \mathcal{L}(X)$ for $r\in\mathcal{I}$. If there exist functions $\kappa,\sigma\colon\mathcal{I}\to(0,\infty)$ and a constant $\Gamma\geq0$ such that
\begin{subequations}\label{eqs:VSC_cond}
\begin{align}
    \label{eq:VSCsmoothnesCondition}
    \Vert f^\dagger-P_rf^\dagger\Vert_X &\leq \kappa(r),\hspace{1em} \inf_{r\in\mathcal{I}}\kappa(r)=0 \\
    \label{eq:VSCillposednessCondition}
    \langle f^\dagger,P_r(f^\dagger-f)\rangle_X &\leq \sigma(r)\Vert F(f^\dagger)-F(f)\Vert_Y + \Gamma\kappa(r)\Vert f^\dagger-f\Vert_X
\end{align}
\end{subequations}
for all $f\in\dom(F)$ with $\Vert f-f^\dagger\Vert_X\leq4\Vert f^\dagger\Vert_X$, then $f^{\dagger}$ satisfies a variational source condition \cref{eq:DefinitionVSC} with the concave index function 
\begin{equation}
\label{VSCfromTheorem}
\Phi (\Delta):=2 \inf_{r\in\mathcal{I}}\Big((\Gamma+1)^2\kappa(r)^2 + \sigma(r)\sqrt{\Delta}\Big)
\end{equation}
(see \cite[Theorem 2.1]{hohage2017}). 

Although the two functions $\kappa$ and $\sigma$ must be identified explicitly, the problem is simplified because \cref{eq:VSCsmoothnesCondition} usually reduces to a classical smoothness condition for $f^\dagger$ when the projections $P_r$ are chosen appropriately for the function space $X$. Therefore, the main difficulty lies in verifying the ill-posedness condition \cref{eq:VSCillposednessCondition}. In particular, we can formulate the following proposition concerning a variational source condition for the random inverse source problem by taking $\dom(F)=H^m(\sourcedom)$ and assuming that $f^\dagger$ belongs to a Sobolev ball in the smoother space $H^s(\sourcedom)\subset H^m(\sourcedom)$.
\begin{proposition} \label{prop:VSC}
Suppose that Assumptions \ref{ass:SourceAssumptions} and \ref{ass:Measurement} are satisfied, and let $0\leq m<s$. Suppose also that $\kappa\geq 1/R$, with $R$ defined in Assumption \ref{ass:Measurement}. 
Let $F=\mathcal{C}$, $X=H^m(\sourcedom)$, and $Y=\HS(L^2(\M))$. If the exact solution $q^\dagger$ satisfies $\Vert q^\dagger\Vert_{H^s}\leq C_s$, then a variational source condition of the form \cref{eq:DefinitionVSC} holds with index function
\begin{align*}
    \Phi (\delta^2)=\phi_{\kappa}(\delta)^2
\end{align*}
and the function $\phi_{\kappa}$ given in \cref{eq:Hoelderlog_stability}.
\end{proposition}

%% file: 04_Proof.tex
\section{Proof of \cref{prop:VSC}}
\label{sec:ProofMainProposition}
As pointed out before, the main difficulty in the verification of the variational source condition in \cref{prop:VSC} is the derivation of the ill-posedness condition \cref{eq:VSCillposednessCondition}. 
The proof uses geometrical optics solutions as in \cite{Alessandrini1988StableDO,HohageWeidling2015}. \emph{Complex geometrical optics solutions} (CGOSs), first introduced by Faddeev \cite{Faddeev1965}, are spatially exponentially growing solutions of the homogeneous partial differential equation under consideration. Their existence is known for many different equations (cf. \cite{Haehner1998,Uhlmann2008}). In the case of the Helmholtz equation \cref{eq:Helmholtz} with a homogeneous background, the complex geometrical optics solutions are given, for each $\xi\in \mathbb{C}^3\setminus \mathbb{R}^3$ such that $\xi\cdot\xi=\kappa^2$, by
\begin{equation}
\label{DefinitionCGOS}
u_\xi(x)=\exp(\imath \xi\cdot x) \text{ for }x\in \mathbb{R}^3.
\end{equation}
Note that these functions resemble plane waves that grow exponentially in the direction of $-\Im(\xi)$. 
The main idea of this method is to choose one direction 
$\imath \xi$ and its conjugate direction so that their product
is a plane wave of arbitrary frequency that can be controlled through the imaginary part of $\xi$. Hence, we can write the integral kernel of the Fourier transform as a product of two such solutions. In this way, CGOSs can be used to bound the Fourier coefficients of the source strength; these bounds will be
used to derive the ill-posedness estimate \cref{eq:VSCillposednessCondition} for the inverse random source problem. 

An essential step in the proof is to show that CGO solutions belong to the range of 
the adjoint $\mathcal{G}^*\colon L^2(\M)\to L^2(\sourcedom)$ of the volume-potential operator
$\mathcal{G}$ 
defined in \eqref{eq:VolumePotentialOperator}. This adjoint is related to the \emph{single-layer potential}  
\begin{align}\label{eq:SLP}
\begin{aligned}
 \SLP\colon H^{-\sfrac{1}{2}}(\M) &\to H^{1}_{\loc}(\RR^3) \\
  (\SLP\phi)(x)&:=\int_{\M}\GreenFunction[(x,y)]\phi(y)\diff y \quad x\in \mathbb{R}^3\setminus\M, 
\end{aligned}
\end{align}
by
\begin{align}\label{eq:Gstar_SLP}
    \left(\mathcal{G}^*\phi\right)(x) = \overline{(\SLP\overline{\phi})(x)}\qquad \text{for }x\in \sourcedom,\, \phi\in L^2(\M)\,.
\end{align}

The following lemma explicitly expresses the CGO solution $u_{\xi}|_{\sourcedom}$ as an element of the 
range of $\mathcal{G}^*$ with the help of the \emph{single-layer operator}
\begin{align*}
 \SLO\colon H^{-\sfrac{1}{2}}(\M) &\to H^{\sfrac{1}{2}}(\M) \\
  (\SLO\phi)(y)&:= \int_{\M}\GreenFunction(x,y)\phi(x)\diff x \quad y\in\M,
\end{align*}
and the normal derivative of the single-layer potential operator
\begin{align*}
    \DLO'\colon H^{-\sfrac{1}{2}}(\M) & \to H^{\sfrac{1}{2}}(\M) \\
    (\DLO'\phi)(y) &:= \int_{\M}\frac{\partial\GreenFunction}{\partial \nu_x}(x,y)\phi(x)\diff x \quad y\in\M.
\end{align*} 
Here $\nu_x$ denotes the exterior unit normal vector to $\mathcal{O}$ at $x\in\M=\partial{\mathcal{O}}$. 

\begin{lemma}\label{lem:CGOSInteriorDirichlet}
Consider the following boundary value problem in the bounded domain $\mathcal{O}$, where $\mathcal{O}$ and its boundary $\M$ are as in Assumption \ref{ass:Measurement}:
\begin{align}
\label{eq:InteriorDirichlet}
\begin{aligned}
- \Delta u - \kappa^2u &= 0 && \text{ in } \mathcal{O}, \\
 \frac{\partial u}{\partial \nu}-i\etaaskappa u&=g && \text{ on } \M.
\end{aligned}
\end{align}
% \todo[inline]{Ggf. gleich $\eta=\kappa$ oder so setzen!}
\begin{enumerate}
    \item\label{it:BVP1} 
    For any $g\in H^{-1/2}(\M)$ and any $\etaaskappa>0$, the boundary value problem \eqref{eq:InteriorDirichlet} has  
    a unique solution given by
    \begin{align}\label{eq:sol_repr}
    u = \SLP\left(\frac{1}{2}I+\DLO'-i\etaaskappa \SLO\right)^{-1}g.
    \end{align}
    \item\label{it:BVP2} 
    For any $\xi\in\mathbb{C}^3$ with $\xi\cdot\xi =\kappa^2$, we have 
    \begin{align}\label{eq:adjoint_repr}
        \begin{aligned}
        &u_{\xi}|_{\sourcedom} = (\SLP \phi_{\xi})_{\sourcedom} = \mathcal{G}^*\overline{\phi_{-\overline{\xi}}}\\
        &\text{with}\quad 
        \phi_{\xi}=\left(\frac{1}{2}I+\DLO'-i\etaaskappa \SLO\right)^{-1}g_{\xi}\quad \text{and}\quad
        g_{\xi} = \frac{\partial u_{\xi}}{\partial \nu}-i\etaaskappa u_{\xi}.
        \end{aligned}
    \end{align}
    \item\label{it:BVP3}
    Suppose that $R \kappa \geq 1$, where $R=\max\{\Vert x\Vert :x\in\M\}$ is defined in
    Assumption \ref{ass:Measurement}. For any $\xi\in\mathbb{C}^3$ with $\xi\cdot\xi =\kappa^2$, let $t:=|\Im(\xi)|$. Then there exists $\phi \in L^2(\M)$ such that $u_\xi|_{\sourcedom}=\mathcal{G}^\ast\phi$ and
    \begin{equation}\label{Gboundphi}
        \Vert \phi\Vert_{L^{2}(\M)}
        \leq C \sqrt{\kappa^2+ t^2}e^{Rt}
    \end{equation}
    holds for some constant $C$ independent of $\xi$ and $\kappa$.
\end{enumerate}
\end{lemma}

\begin{proof}
\emph{Part \ref{it:BVP1}.}
The weak formulation of \cref{eq:InteriorDirichlet} is 
\[
\int_{\mathcal{O}} \nabla u \overline{\nabla v} - \kappa^2u\overline{v}\diff x 
+ i\etaaskappa \int_{\M} u \overline{v} \diff x
= -\int_{\M} g \overline{v}\diff x\quad \text{for all } v\in H^{1}(\mathcal{O}).
\]
Choosing $v=u$ for $g=0$ and taking the imaginary part shows that $u=0$ on $\M$ in the sense of the trace operator. 
By elliptic regularity results (see \cite[Theorem 6.15]{GT:77}) and the assumed boundary regularity, 
$u$ belongs to the H\"older space $C^{2,\alpha}(\overline{\mathcal{O}})$. 
Hence, the normal derivative $\frac{\partial u}{\partial \nu}$ exists on $\M$. 
Integration by parts shows that $\frac{\partial u}{\partial \nu}=0$.
It now follows from Holmgren's uniqueness theorem or, more simply, from Green's representation formula 
\cite[Theorem 2.3]{ColtonKressIAaES} that $u$ must vanish in $\mathcal{O}$. 

Because the boundary $\M$ is assumed to be $C^2$-smooth, the boundary integral operators 
$\SLO,\DLO'\colon H^{-1/2}(\M)\to H^{1/2}(\M)$ are bounded \cite[Corollary 3.7]{ColtonKressIAaES}.  
By the jump relations \cite[Theorem 3.1]{ColtonKressIAaES} and the properties of Green's function, $\SLP \phi$ is a solution of \cref{eq:InteriorDirichlet} if $\phi$ satisfies 
the boundary integral equation
\begin{align}\label{eq:BIE}
\left(\frac{1}{2}I+\DLO'-i\etaaskappa \SLO\right)\phi = g.
\end{align}
It follows from the uniqueness of \cref{eq:InteriorDirichlet} that the operator on the left-hand side 
of this equation is injective. By Riesz theory and the compactness of $\SLO$ and $\DLO'$, it is also 
surjective. This implies the existence of a solution to \eqref{eq:InteriorDirichlet} and the representation 
\cref{eq:sol_repr}.

\emph{Part \ref{it:BVP2}.} 
The identity $u_{\xi}|_{\sourcedom} = \SLP \phi_{\xi}$ follows immediately from the first part because 
$u_{\xi}$ solves \cref{eq:InteriorDirichlet} with $g=g_{\xi}$ under our assumption $\xi\cdot \xi = \kappa^2$. 
Note that $\overline{u_\xi}=u_{-\overline{\xi}}$. Hence, together with \cref{eq:Gstar_SLP} we obtain
\[
u_{\xi}|_{\sourcedom} = \overline{u_{-\overline{\xi}}}|_{\sourcedom}
=\overline{\SLP \phi_{-\overline{\xi}}} = \mathcal{G}^*\overline{\phi_{-\overline{\xi}}}.
\]

\emph{Part \ref{it:BVP3}.} 
Under Assumption \ref{ass:Measurement}, \cite[Corollary 4.4]{chandler-wilde_wave-number-explicit_2008} shows that, for $R \kappa \geq 1$, the inverse operator satisfies the wave-number independent bound
\begin{align}\label{eq:uniform_iop_bound}
\left\|\left(\frac{1}{2}I+\DLO'-i\etaaskappa \SLO\right)^{-1}\right\|_2 \leq \frac{1}{2} (1+ \theta(4\theta + 13)), 
\end{align}
where we have adapted the bound to dimension $n=3$ with a geometric constant 
$\theta:=\frac{R}{\delta_{\M}}$ as in \cite{chandler-wilde_wave-number-explicit_2008}. By Part \ref{it:BVP2}, it remains to bound
$\| g_\xi\|_{L^2(\M)}$ for $g_\xi = \frac{\partial u_\xi}{\partial \nu} - i \kappa u_\xi$. 
A straightforward computation shows that 
\begin{equation*}
    \| g_\xi\|_{L^2(\M)} \leq |\M|^{1/2} (|\xi| + \kappa) e^{Rt} 
    = |\M|^{1/2} \left(\left(\kappa^2+ 2t^2\right)^{\frac{1}{2}} + \kappa\right)e^{Rt}.
\end{equation*}
Because $\M$ is the boundary of a star-shaped domain $\mathcal{O}$ and $\delta_{\M}>0$,  
it admits a radial parameterization $S^2\to \M$, $\widehat{x}\mapsto \pi(\widehat{x})\widehat{x}$, with
a $C^{2,\alpha}$-smooth radial function $\pi \colon S^2 \to \RR_{+}$. 
Because $S^2$ is compact, the gradient of $\pi$ is bounded. 
This implies that the surface area $|\M|$ is finite. 
Combining this with \cref{eq:uniform_iop_bound} yields \cref{Gboundphi} with 
 $C=(1+\sqrt{2})|\M|^{1/2}$.
\end{proof}

In the next lemma, we prove bounds on the Fourier coefficients of the difference between two source strengths. For the Fourier transform $\mathcal{F}\colon L^2(\mathbb{R}^3) \to L^2(\mathbb{R}^3)$, we use the convention 
\begin{align}
    \label{eq:DefinitionFourierTransform}    
    (\mathcal{F}f)(\gamma):= \frac{1}{(2\pi)^{3/2}} \int_{\mathbb{R}^3} f(x)e^{-i\gamma\cdot x}\diff x, \qquad 
    \gamma\in \mathbb{R}^3.
\end{align}
For $f\in L^2(\RR^3)$, we use the notation 
$\hat{f}:= \mathcal{F}f$. 
\begin{lemma}\label{lem:FourBound}
Suppose that Assumptions \ref{ass:SourceAssumptions} and \ref{ass:Measurement} are satisfied. Let $q_1,q_2 \in L^\infty(\sourcedom)$ have corresponding covariance operators $\mathcal{C}_{q_1},\allowbreak \mathcal{C}_{q_2} \in \HS(L^2(\M))$. 
Let $\rho,t\in (0,\infty)$ such that
\begin{align}\label{eq:constraint}
\rho \leq 2\sqrt{\kappa^2+t^2}.
\end{align}
Then there exists a constant $C>0$, depending only on $\M$, such that 
\begin{equation}\label{FourierCoeffitientBound}
 |(\hat{q}_1-\hat{q}_2)(\gamma)|\leq C\left(\kappa^2+ t^2\right)e^{2Rt}\Vert \mathcal{C}_{q_1}-\mathcal{C}_{q_2}\Vert_{\HS},
\end{equation}
for all $\gamma\in\RR^3$ with $|\gamma|\leq\rho$, where $R=\max\{\Vert x\Vert :x\in\M\}<\infty$.
\end{lemma}

\begin{proof}
 Extending $q_1,q_2$ to $\mathbb{R}^3$ by zero, their Fourier transforms satisfy 
 \begin{equation}\label{FourierDef}
  |(\hat{q}_1-\hat{q}_2)(\gamma)|=\frac{1}{(2\pi)^{3/2}}\Big|\int_{\mathbb{R}^3}(q_1-q_2)e^{-i\gamma\cdot x}\diff x\Big|,
  \qquad \gamma\in\mathbb{R}^3.
 \end{equation}
 Choose two unit vectors $d_1,d_2\in\mathbb{R}^3$ such that $d_1,d_2,\gamma$ are mutually orthogonal and define
\begin{align*}
 a&:=-\frac{1}{2}\gamma+\imath td_1+\sqrt{\kappa^2+t^2-\frac{|\gamma|^2}{4}}d_2, \\
 b&:=\frac{1}{2}\gamma-\imath td_1+\sqrt{\kappa^2+t^2-\frac{|\gamma|^2}{4}}d_2.
\end{align*}
Note that the argument of the square root is nonnegative by \cref{eq:constraint} and that
\[a-\Bar{b}=-\gamma\hspace{1cm}|\Im(a)|=|\Im(b)|=t\hspace{1cm}a\cdot a=b\cdot b=\kappa^2.\]
We get $e^{-i\gamma\cdot x}=u_{a}(x)\overline{u_b(x)}$, and substituting this into \cref{FourierDef} gives
\begin{align} 
(2\pi)^{3/2} |(\hat{q}_1-\hat{q}_2)(\gamma)| &=\Big|\int_{\mathbb{R}^3}(q_1-q_2)u_a\bar{u_b}\diff x\Big| 
 =\big|\big\langle(q_1-q_2)u_a,u_b \big\rangle_{L^2(\mathbb{R}^3)} \big|\nonumber \\
\label{eq:CGOSReplacement}
 &=\big|\big\langle(q_1-q_2)u_a,u_b \big\rangle_{L^2(\sourcedom)} \big|.
\end{align}
In the last equality, we used the fact that $q_1$ and $q_2$ are supported in $\sourcedom$.
By part \ref{it:BVP3} of \cref{lem:CGOSInteriorDirichlet}, there exist $\phi_a,\phi_b\in L^2(\M)$ such that $\mathcal{G}^\ast\phi_a =u_a$ and $\mathcal{G}^\ast\phi_b =u_b$ in $\sourcedom$. 
Substitution into \cref{eq:CGOSReplacement} yields
\begin{align*}
 (2\pi)^{3/2}|(\hat{q}_1-\hat{q}_2)(\gamma)| 
 &=  \big| \big\langle(q_1-q_2)\mathcal{G}^\ast\phi_a,\mathcal{G}^\ast\phi_b \big\rangle_{L^2(\sourcedom)}  \big|\\ 
& =  \big| \big\langle\mathcal{G}M_{q_1-q_2}\mathcal{G}^\ast\phi_a,\phi_b \big\rangle_{L^2(\M)}\big|.
\end{align*}
Applying the Cauchy-Schwarz inequality and the bound \cref{Gboundphi} from the previous lemma yields the assertion \cref{FourierCoeffitientBound}.
\end{proof}

\begin{remark}
  \label{rem:uniqueness}
The lemma immediately implies uniqueness. If $q_1,q_2\in L^\infty(\sourcedom)$ are such that the corresponding data satisfy $\Cq[q_1]=\Cq[q_2]\in \HS(L^2(\M))$, then \cref{FourierCoeffitientBound} yields
\[|(\hat{q}_1-\hat{q}_2)(\gamma)|=0\]
for $\gamma \in B(\rho)$, i.e., the Fourier transform of $q_1-q_2$ vanishes on an open set. By analytic continuation, it therefore vanishes throughout $\mathbb{R}^3$. This yields $q_1=q_2$ in $\sourcedom$ and hence uniqueness.
\end{remark}

\Cref{lem:FourBound}, together with a Sobolev-type smoothness assumption on the true solution, enables us to prove \cref{prop:VSC}. Specifically, we choose the domain $H^m(\sourcedom)=\dom(\COperator)$ for $m\geq 0$, with the canonical embedding into $L^2(\sourcedom)$, and assume that the true solution $q^\dagger$ belongs to $H^s$, where $m<s$. 

\Cref{lem:FourBound} bounds the low-frequency Fourier coefficients of differences between sources in terms of the Hilbert--Schmidt norm of the corresponding covariance operators. A natural choice for the family of projections 
in \cref{eqs:VSC_cond} is given by the low-frequency filters 
\begin{align*}
P_\rho\colon H^m&\to H^m \\
P_\rho f&:=\mathcal{F}^\ast\phi_\rho \mathcal{F} f
\end{align*}
for $\rho>0$, where $\phi_\rho$ is the cutoff function given by
\[
\phi_\rho(\omega)=
\begin{cases}
1 \hspace{1em} \omega\in B(\rho) \\
0 \hspace{1em} \omega\notin B(\rho)
\end{cases}.
\]
Note that the Fourier transform defined in \cref{eq:DefinitionFourierTransform} extends to $H^m$ in the usual way. 

\begin{proof}[Proof of \cref{prop:VSC}]
We follow the approach outlined in \Cref{sec:VSC} and verify the conditions in 
\cref{eqs:VSC_cond} to obtain a variational source condition with an index function given 
implicitly by an infimum and then derive an explicit estimate. 

\emph{Step 1: Verification of the generalized smoothness condition of the exact solution \cref{eq:VSCsmoothnesCondition}.}
This condition follows from the Sobolev smoothness assumption on the exact solution:
\begin{align*}
\Vert (I-P_\rho)q^\dagger\Vert_{H^m}^2 &= \int_{\mathbb{R}^3} (1+|\gamma|^2)^m|\mathcal{F}((I-P_\rho)q^\dagger)(\gamma)|^2\diff \gamma \\
&\leq \frac{1}{(1+\rho^2)^{s-m}}\int_{\mathbb{R}^3\setminus B(\rho)}\big(1+|\gamma|^2\big)^s|\hat{q}^\dagger(\gamma)|^2 \\
&\leq\rho^{2(m-s)}\Vert q^\dagger\Vert_{H^s}^2
\end{align*}
for every $\rho>0$. Thus, choosing
\begin{equation*}
\kappa(\rho)  :=\rho^{m-s}\Cs
\end{equation*}
yields \cref{eq:VSCsmoothnesCondition} because $\inf_{\rho}\kappa(\rho)=0$ for $m<s$.

\emph{Step 2: Verification of the local degree of ill-posedness condition \cref{eq:VSCillposednessCondition}.}
We rewrite the left-hand side of \cref{eq:VSCillposednessCondition} as
\begin{equation}
    \label{eq:LocDegIllPos}
    \Re\langle P_\rho q^\dagger,q^\dagger-q\rangle_{H^m}=\Re\int_{B(\rho)}\big(1+|\gamma|^2\big)^m\hat{q}^\dagger(\gamma) \big(\overline{\hat{q}^\dagger(\gamma)-\hat{q}(\gamma)}\big) \diff \gamma.
\end{equation}
We have
\begin{equation}
\label{eq:Inequality}
\int_{B(\rho)}(1+|\gamma|^2)^m |\hat{q}^\dagger(\gamma) | \diff \gamma
\leq \tilde{c} \rho^\tau
\end{equation}
with $\tau=\max\{2m+\frac{3}{2}-s,0\}$ 
by a straightforward adaptation of the analogous 
estimate in \cite[Lemma 4.3]{HohageWeidling2015} 
for the discrete Fourier transform. 
Together with the bound \cref{FourierCoeffitientBound} on 
$|\hat{q}^\dagger(\gamma)-\hat{q}(\gamma)|$ 
we obtain the inequality
\begin{align*}
	\Re\langle P_\rho q^\dagger,q^\dagger-q\rangle_{H^m}  
	\leq  &C\left(\kappa^2+ t^2\right)e^{2Rt}\Vert \mathcal{C}_{q^\dagger}-\mathcal{C}_{q}\Vert_{\HS} \tilde{c}\rho^\tau.
\end{align*}
Therefore, the ill-posedness condition \cref{eq:VSCillposednessCondition} holds true with 
$\Gamma=0$ and
\begin{equation*}
    \sigma(\rho):=C\left(\kappa^2+ t^2\right)e^{2Rt}\tilde{c}\rho^\tau.
\end{equation*} 

\emph{Step 3: Estimation of the infimum.}
The first two steps of the proof and \cite[Theorem 2.1]{hohage2017} yield a variational 
source condition. Taking into account \cref{eq:constraint} in \cref{lem:FourBound},
we obtain the following expression for the index function in \cref{VSCfromTheorem}:
\begin{equation}
    \label{eq:concaveIndexFunction}
    \Phi(\delta^2):=\inf_{\rho,t>0,\rho\leq2\sqrt{\kappa^2+t^2}} \left[C\left(\kappa^2+ t^2\right)e^{2Rt}\delta c_4\rho^\tau +\frac{16}{3}\rho^{2(m-s)}\Cs^2
    \right].
\end{equation}
To derive an explicit estimate of this infimum, 
we choose the parameters $\rho$ and $t$ as functions of the noise level $\delta$ such that the side constraints are satisfied. A logarithmic variational source condition with the index function in \cref{eq:ExplicitStabilityEstimate} can be derived directly 
by choosing
\begin{equation}\label{eq:ChoiceTRho}
6Rt=\log(3+\delta^{-2})=3R\rho
\end{equation}
Then $\rho=2t$, and the side condition $\rho\leq 2\sqrt{\kappa^2+t^2}$ is satisfied. Hence, $\Phi$ 
in \cref{eq:concaveIndexFunction} is bounded by
\begin{align*}
\Phi(\delta^2)&\leq
    \frac{16}{3}(3R)^{2(s-m)}\Big(\log(3\!+\!\delta^{-2})\Big)^{-2(s-m)}\Cs^2\left(1\!+\!o(1)\right) 
    \\&
    \leq C^2\Big(\log(3\!+\!\delta^{-2})\Big)^{-2(s-m)}
\end{align*}
with a constant $C$ depending on $C_s,\M,\kappa$, and $R$.

We modify the choice in \cref{eq:ChoiceTRho} as follows to include the wave number $\kappa$ explicitly and ensure that the 
side constraint is satisfied with equality, i.e., 
$(\rho/2)^2=\kappa^2+t^2$:
\begin{equation}
    \label{eq:ChoiceTRhoForHoelder}
    (6Rt)^2 = \left(\log(3+\delta^{-2})\right)^2 = (3R)^2(\rho^2-4\kappa^2).
\end{equation}
This leads to $\rho = \overline{\rho}(\delta,\kappa)$ 
and the inequality $\Phi(\delta^2)\leq \phi_{\kappa}(\delta)^2$
with $\overline{\rho}$ and $\phi_{\kappa}$ defined 
in \cref{eq:Hoelderlog_stability}. 
\end{proof}

% \left( \left(\kappa^2\!+\!2t^2\right)^{\sfrac{1}{2}}\!+\!\kappa\right)^2
% = (2kappa^2 + 2t^2 + 2kappa(kappa^2+2t^2)^(1/2)) 
% = (2(kappa^2 + t^2) + 2 (kappa^4 +2kappa^2t^2)^(1/2)))
% <=(2(kappa^2 + t^2) + 2 (kappa^4 +2kappa^2t^2 + t^4)^(1/2)))
% = (2(kappa^2 + t^2) + 2 (kappa^2 +t^2))
% = 4*(kappa^2 + t^2) 
% = 4 * (rho/2)^2 
% = rho^2

% --> 
% C\left( \left(\kappa^2\!+\!2t^2\right)^{\sfrac{1}{2}}\!+\!\kappa\right)^2e^{2Rt}\delta c_4\rho^\tau
% <= C rho^(2+tau) (3+delta^-2)^(1/3) delta

%% file: 05_NumericalParts.tex
\section{Numerical tests of convergence rates}\label{sec:Numerical}
In this section, we present a numerical analysis of the convergence rates and compare the results with the rates established in \cref{thm:ConvergenceRate}.  

\paragraph*{Numerical setup} We first present the discretization of the sources in the domain. We assume that the sources are supported in the cube $\sourcedom=\Big[-\frac{\pi}{\sqrt{3}},\frac{\pi}{\sqrt{3}}\Big]^3$.
Functions $f\colon\sourcedom\to [0,\infty)$ are represented by their values on a uniform rectangular grid of $60\times60\times60$ points. To study the effect of the smoothness of the exact solution on the convergence rates, we construct exact source strengths $q^{\dagger}$ as multivariate splines of degree $1$ or $3$. More precisely, let $\{b_1^\lambda,\dots,b_n^\lambda\}$ be the set of B-splines of degree $\lambda\in \{1,3\}$ on $[-\frac{\pi}{\sqrt{3}},\frac{\pi}{\sqrt{3}}]$ associated with equidistant knots that divide the interval into 12 subintervals. We define a basis of functions $\{f^\lambda_{k,l,m}\}_{k,l,m\in\{1,\dots,n\}}$ on the cube $\Big[-\frac{\pi}{\sqrt{3}},\frac{\pi}{\sqrt{3}}\Big]^3$ by 
\begin{equation*}
    f^\lambda_{k,l,m}(x_1,x_2,x_3):=b^\lambda_k(x_1)b^\lambda_l(x_2)b^\lambda_m(x_3) 
    = (b^\lambda_k\otimes b^\lambda_l\otimes b^\lambda_m)(x_1,x_2,x_3).
\end{equation*}
Note that $b^{\lambda}_m\in H^{\lambda+\frac{1}{2}-\epsilon}(\mathbb{R})$ for all $\epsilon>0$. (For $\lambda=0$ this follows from explicit computations, and for $\lambda\in\mathbb{N}$ we use the fact that
$\frac{d^{\lambda}}{dx^{\lambda}}b^{\lambda}_m$ is piecewise constant.) It follows that $f^{\lambda}_{k,l,m}\in H^{\lambda+\frac{1}{2}-\epsilon}(\sourcedom)$.
We can represent each multivariate spline in this basis by spline coefficients $(c^\lambda_{k,l,m})_{k,l,m}\in\mathbb{R}^{n\times n\times n}$, where $n$ is the dimension of the one-dimensional spline space. Moreover, because B-splines are nonnegative, $q^{\dagger}:=\sum_{k,l,m}c^{\lambda}_{k,l,m} f^\lambda_{k,l,m}$ is nonnegative whenever all $c^{\lambda}_{k,l,m}$ are nonnegative.
In our experiments, we use two types of $q^{\dagger}$ for each of $\lambda=1$ and $\lambda=3$. For the first type, the coefficients are generated randomly such that $c^{\lambda}_{k,l,m}\sim |X^{\lambda}_{k,l,m}|$, where the $X^{\lambda}_{k,l,m}\sim N(0,1)$ are independent and identically distributed. The functions of the second type approximate a half-sphere and a dot,  
as illustrated in \cref{fig:ExactSolLinear,fig:ExactSolCubic}.

Measurements are taken on $\M=S_R$, a sphere of radius $R$, and functions in $L^2(S_R)$ are expanded in the standard orthonormal basis of spherical harmonics. This representation is convenient because the single-layer potential $\SLP$ in \cref{eq:SLP} can be expressed explicitly in this basis in terms of spherical Hankel and Bessel functions. The same holds for the volume potential $\mathcal{G}$ and its adjoint $\mathcal{G}^*$, because the latter essentially coincides with $\SLP$. The codomain of the forward operator is the space of Hilbert--Schmidt operators on $L^2(S_R)$, which are represented by matrices in the spherical-harmonic basis.

Synthetic data are generated from sample autocorrelations of solutions to the Helmholtz problem driven by random source realizations with variance given by the exact solution $q^\dagger$. More precisely, let $\underline{q}^\dagger$ denote the discrete representation of the exact solution $q^\dagger$ described above, and let $\underline{\mathcal{G}}$ denote the discretization of the volume-potential operator $\mathcal{G}$ that maps $\underline{q}^\dagger$ to the spherical-harmonic coefficients. The discrete noisy data for the inverse problem are then given by
\begin{equation*}
    \underline{C}_{\text{obs}}:= \frac{1}{N} \sum_{i=1}^N \underline{\mathcal{G}}\xi_i (\underline{\mathcal{G}}\xi_i)^\ast
    \quad \text{where}\quad \xi_i\stackrel{\mathrm{i.i.d.}}{\sim} \mathcal{N}_{\mathbb{C}}(0,\underline{q}^\dagger)
\end{equation*}
where $N$ denotes the number of samples. The discrete exact solution is given by 
\begin{equation*}
    \underline{C}^\dagger = \underline{\mathcal{G}} \diag(\underline{q}^\dagger) \underline{\mathcal{G}}^\ast.
\end{equation*}

The numerical setup involves the following four parameters: 
\begin{enumerate}[label=(\roman*)]
    \item the measurement distance, i.e., the radius $R$ of the measurement sphere, which is fixed at $R=4$,
    \item the wave number $\kappa$,
    \item the smoothness of the exact solution $q^\dagger$ 
    determined by the spline degree $\lambda$,
    \item the smoothness of the reconstructions, determined by the index $m$ of the squared Sobolev 
    norm $\|\cdot\|_{H^m}^2$ used as  
    a penalty term in the Tikhonov functional.
\end{enumerate}
For each tuple $(R,\kappa,\lambda,m)$ we
generate synthetic data for 18 different values of \(N\), ranging from \(50\) to \(95000\).  
This leads to different noise levels proportional to $N^{-1/2}$. For reconstruction, we apply classical Tikhonov regularization 
\begin{equation}
    \label{eq:TikhonovReg}
    \hat{q}_\alpha := \argmin_{q\in H^m} \left[\frac{1}{2}\Vert\mathcal{C}(q) - C_{\text{obs}}\Vert_{\HS(L^2(S_R))}^2 + \frac{\alpha}{2} \Vert q\Vert_{H^m(D)}^2\right],
\end{equation}
using the conjugate-gradient (CG) method to find the minimizer. 
We sequentially lower the regularization parameter until a discrepancy principle is satisfied. 

\paragraph*{Discussion of the results} 
Note the following features of the reconstructions displayed in \cref{fig:ReconLinearSmallN,fig:ReconLinearLargeN} and 
\cref{fig:ReconCubicSmallN,fig:ReconCubicLargeN} for 
the exact source strengths in \cref{fig:ExactSolLinear,fig:ExactSolCubic}, respectively:
\begin{itemize}
\item For the comparatively small sample size of $N=550$, the main features of $q^{\dagger}$ are clearly visible for the second type of source (see \cref{fig:ReconLinearSmallN,fig:ReconCubicSmallN}). For the highly varying source of the first type, however, this sample size appears insufficient (see \cref{fig:CompareRandomSampleNumberComparison}).
\item
For the same number of samples, the reconstructions are better for the smoother solutions $q^{\dagger}$, 
as observed in \cref{fig:ReconCubicSmallN,fig:ReconLinearSmallN} and \cref{fig:ReconCubicLargeN,fig:ReconLinearLargeN}, respectively.
\item
Reconstructions improve as the wave number $\kappa$ increases (equivalently, as the wavelength decreases). This effect is discussed in more detail below.
\end{itemize}
\begin{figure}[htb]
\centering
    \begin{subfigure}[t]{0.32\textwidth}
        \centering
        \includegraphics[width=0.75\linewidth]{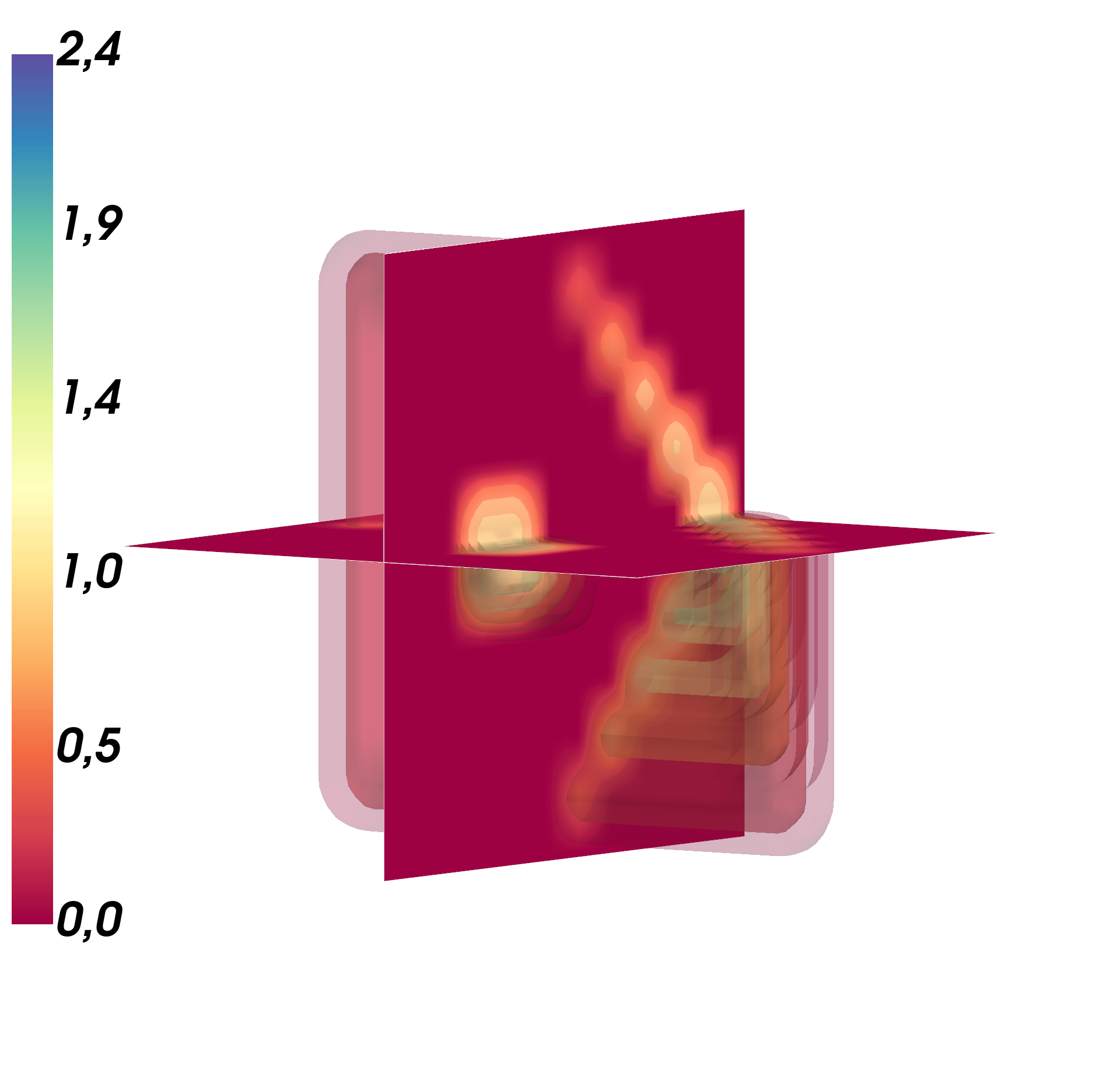}
        \caption{$q^{\dagger}$, multivar. linear spline}
        \label{fig:ExactSolLinear}
    \end{subfigure}
    \begin{subfigure}[t]{0.32\textwidth}
        \centering
        \includegraphics[width=0.75\linewidth]{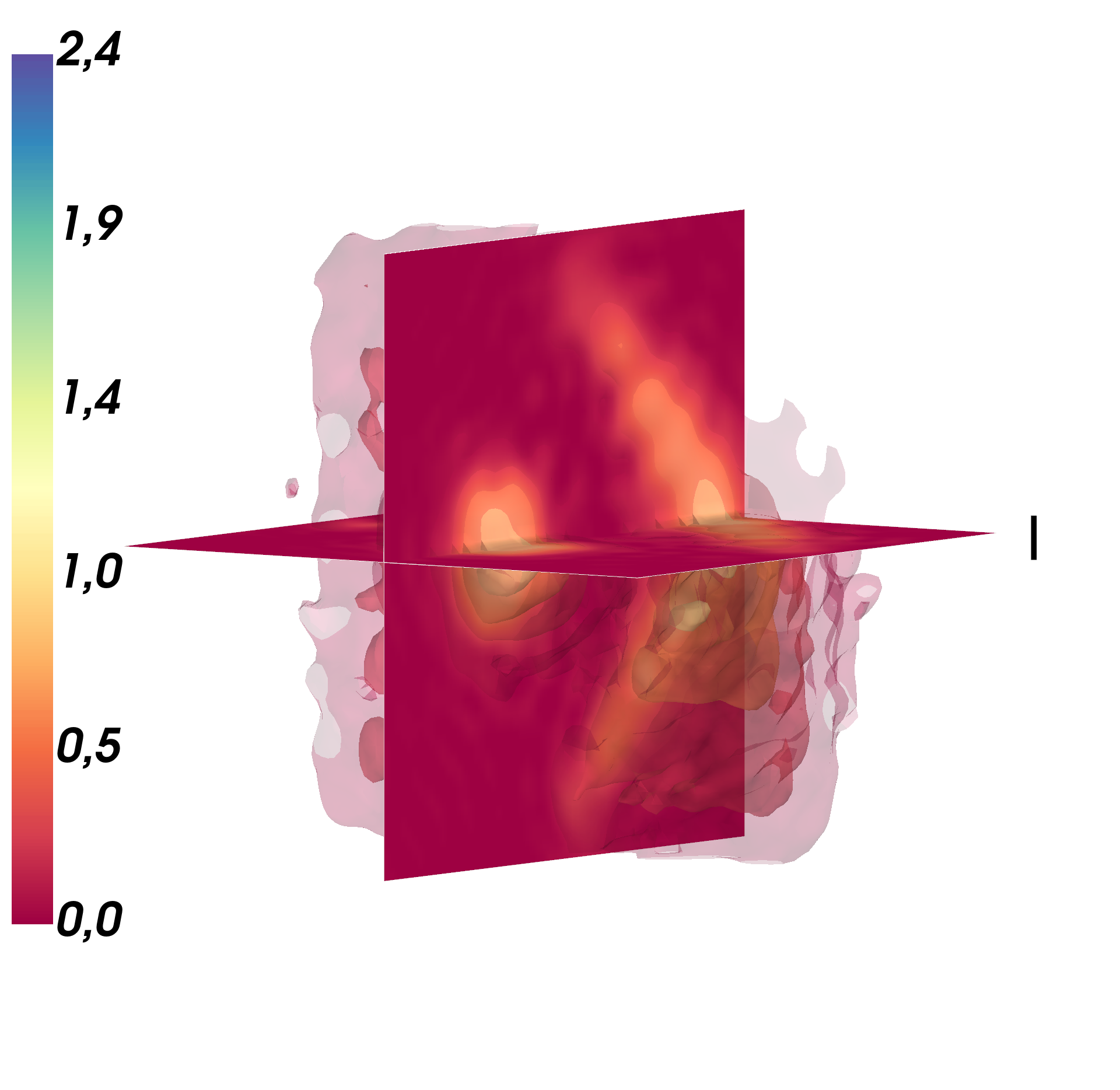}
        \caption{$N=550$ }
        \label{fig:ReconLinearSmallN}
    \end{subfigure}
    \begin{subfigure}[t]{0.32\textwidth}
        \centering
        \includegraphics[width=0.75\linewidth]{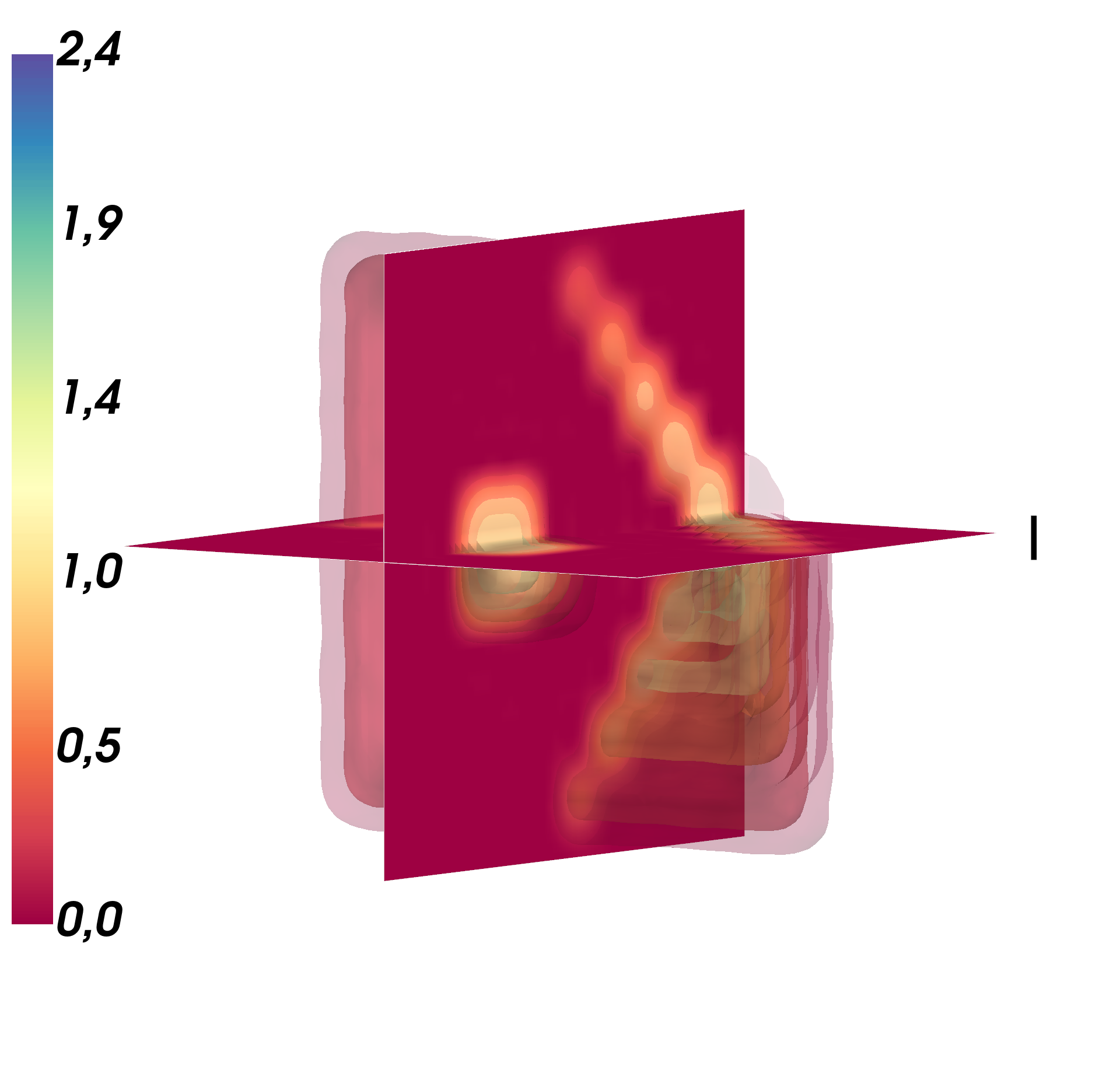}
        \caption{$N=95000$ }
        \label{fig:ReconLinearLargeN}
    \end{subfigure}
\newline
\centering
    \begin{subfigure}[t]{0.32\textwidth}
        \centering
        \includegraphics[width=0.75\linewidth]{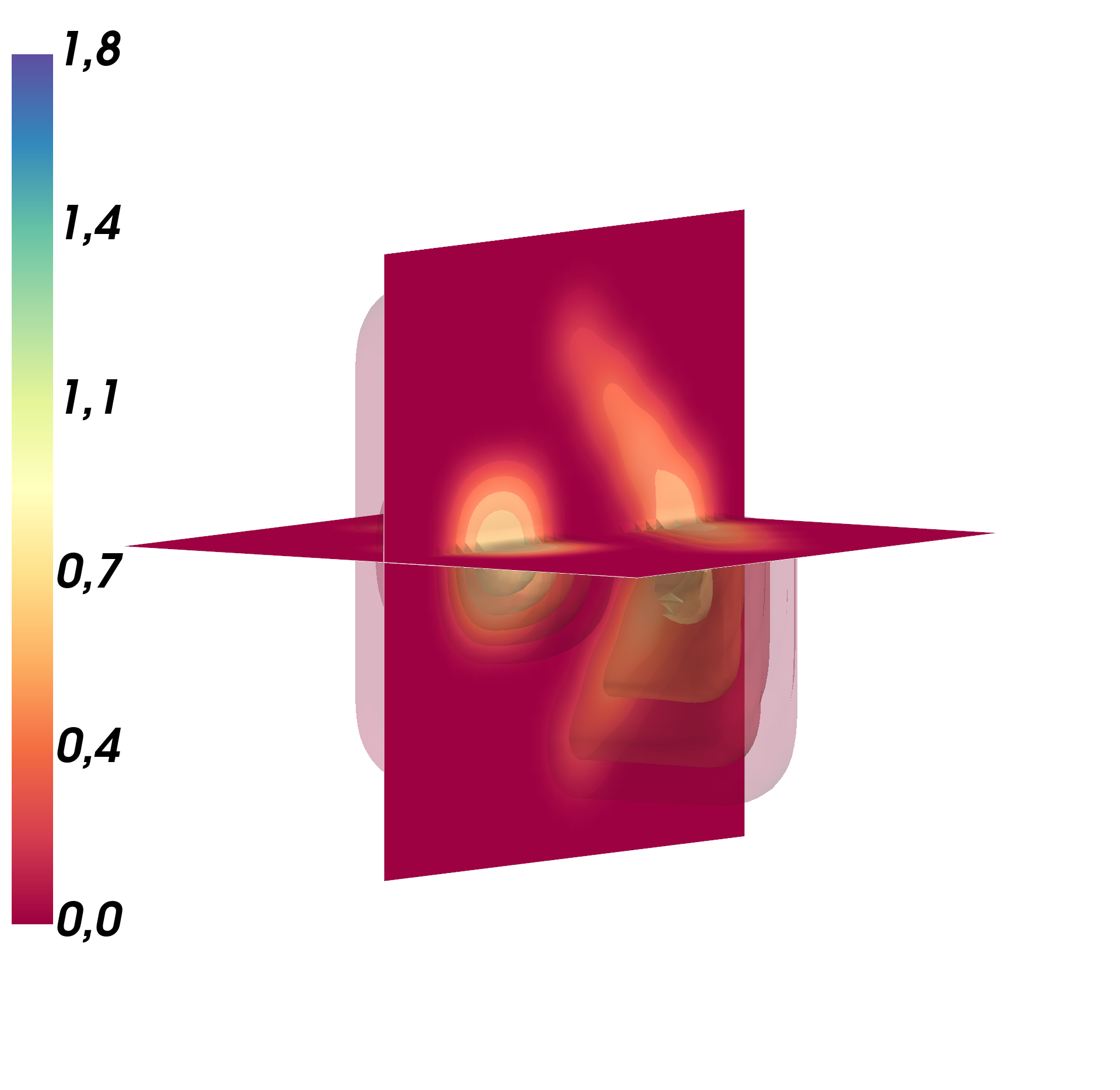}
        \caption{$q^{\dagger}$, multivar. cubic splines}
        \label{fig:ExactSolCubic}
    \end{subfigure}
    \begin{subfigure}[t]{0.32\textwidth}
        \centering
        \includegraphics[width=0.75\linewidth]{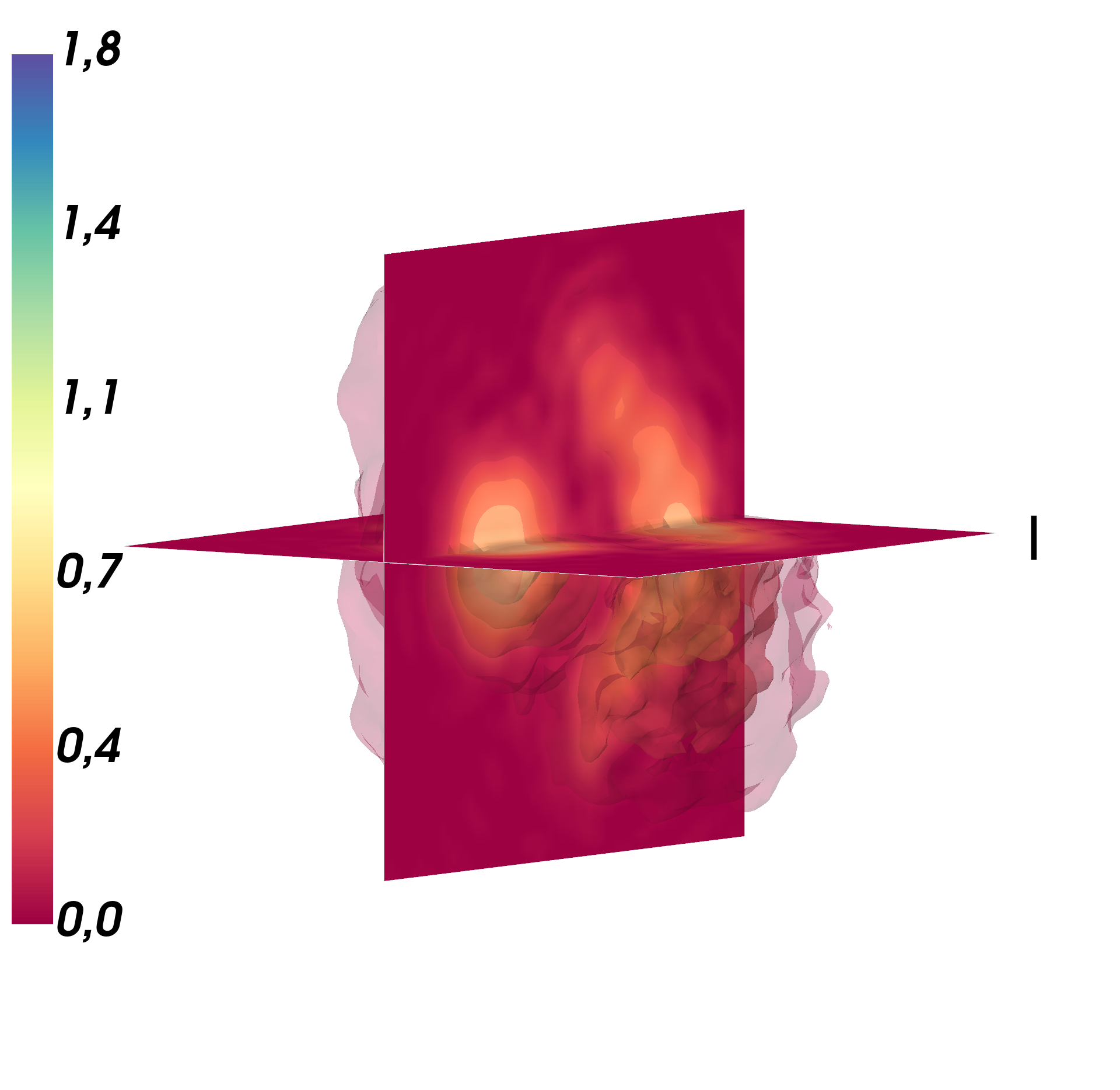}
        \caption{$N=550$ }
        \label{fig:ReconCubicSmallN}
    \end{subfigure}
    \begin{subfigure}[t]{0.32\textwidth}
        \centering
        \includegraphics[width=0.75\linewidth]{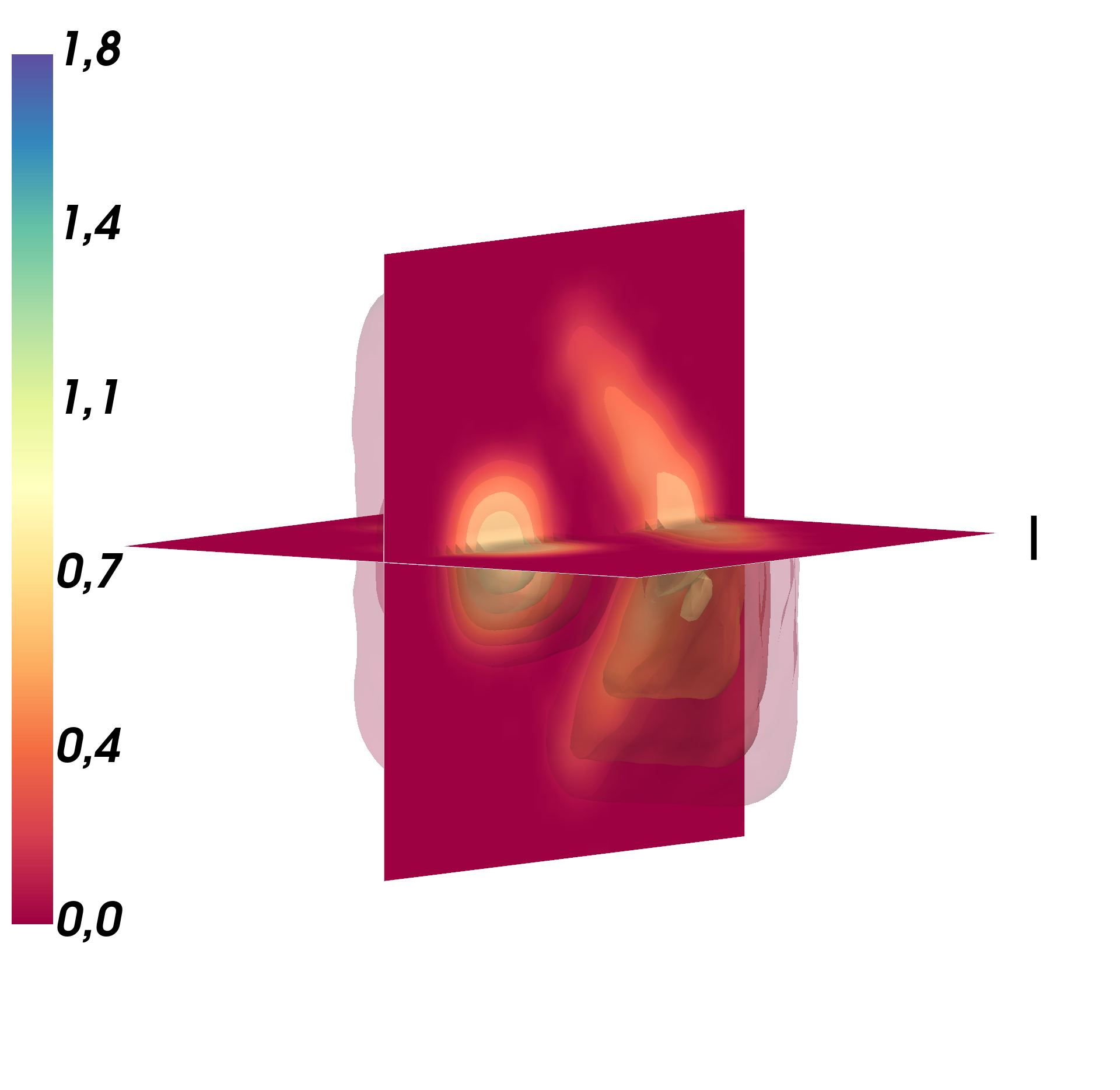}
        \caption{$N=95000$ }
        \label{fig:ReconCubicLargeN}
    \end{subfigure}
    \caption{Each row shows, on the left, the exact solution $q^\dagger$, given by multivariate splines of degree $1$ (top) and $3$ (bottom) that approximate a half-sphere and a hollow dot. The center and right panels show reconstructions for the respective exact solutions, using $\Vert \cdot\Vert_{L^2}$ as the regularization penalty, a fixed wave number $\kappa=12$, and different sample sizes $N$.}
    \label{fig:CompareSampleNumber}
\end{figure}
\begin{figure}[htb]
    \centering
        \begin{subfigure}[t]{0.32\textwidth}
            \centering
            \includegraphics[width=0.75\linewidth]{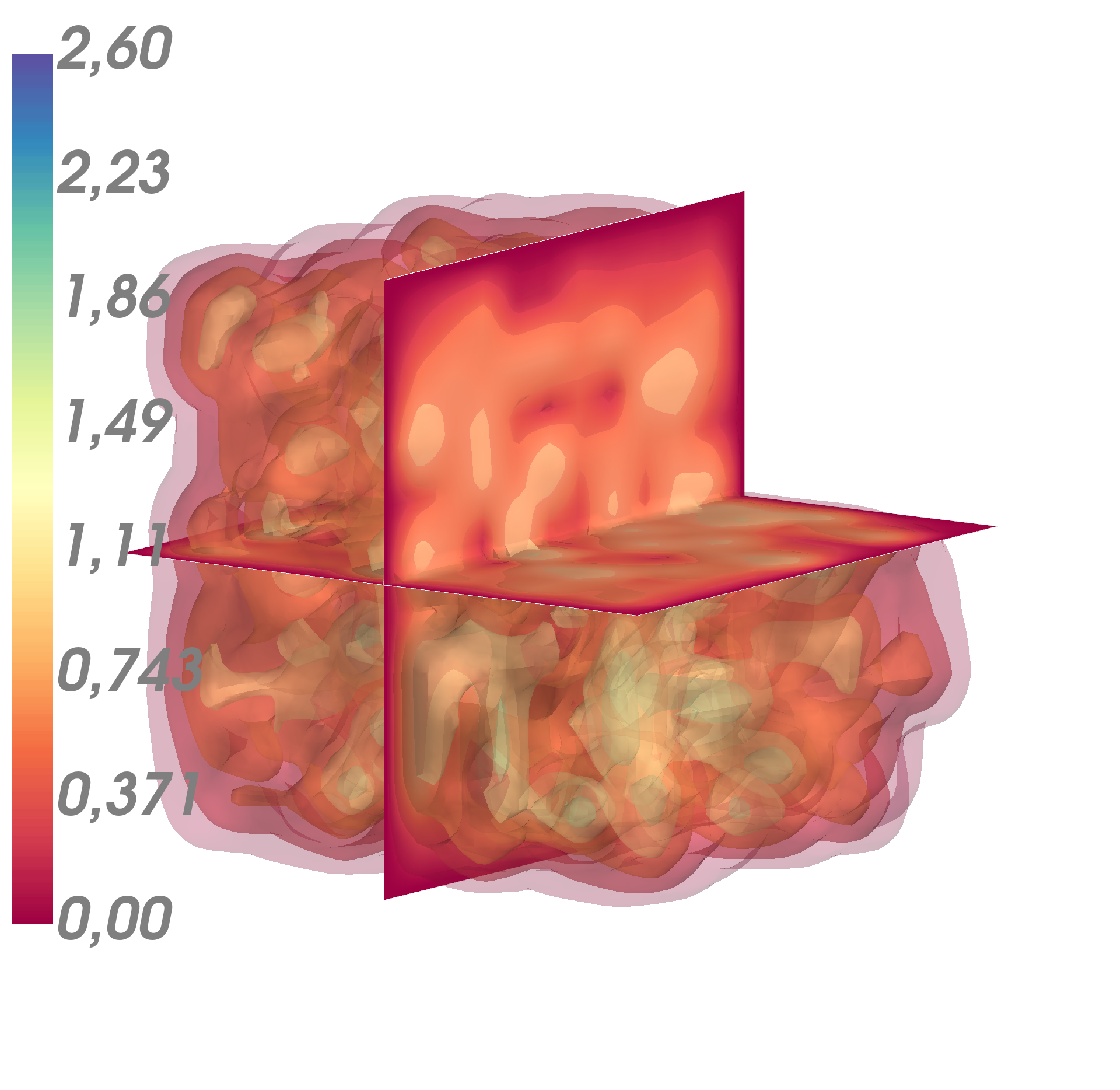}
            \caption{Exact solution $q^{\dagger}$}
            \label{fig:ExactSolCubicRand}
        \end{subfigure}
        \begin{subfigure}[t]{0.32\textwidth}
            \centering
            \includegraphics[width=0.75\linewidth]{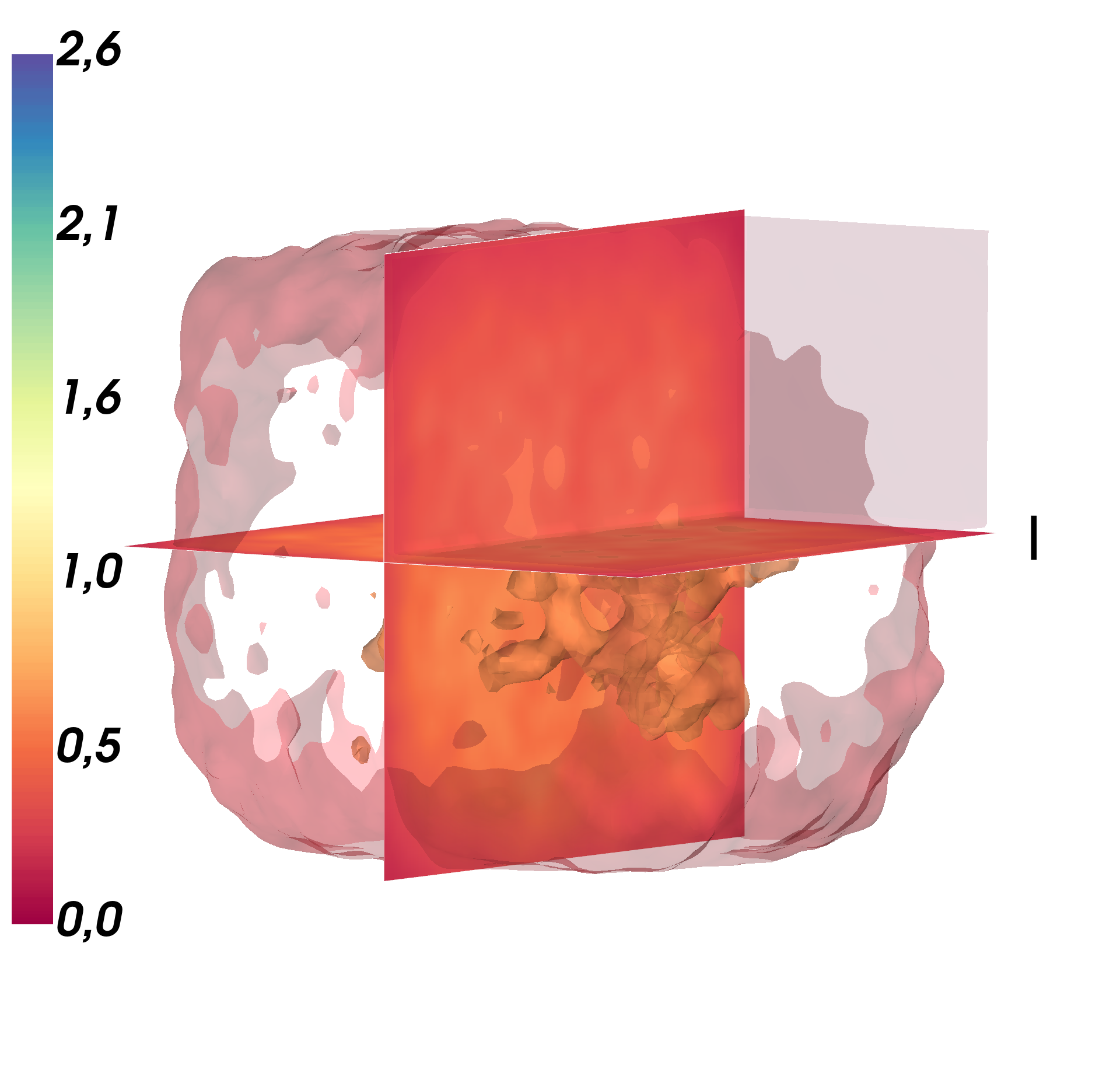}
            \caption{$N=550$}
            \label{fig:ReconLinearRandomSmallN}
        \end{subfigure}
        \begin{subfigure}[t]{0.32\textwidth}
            \centering
            \includegraphics[width=0.75\linewidth]{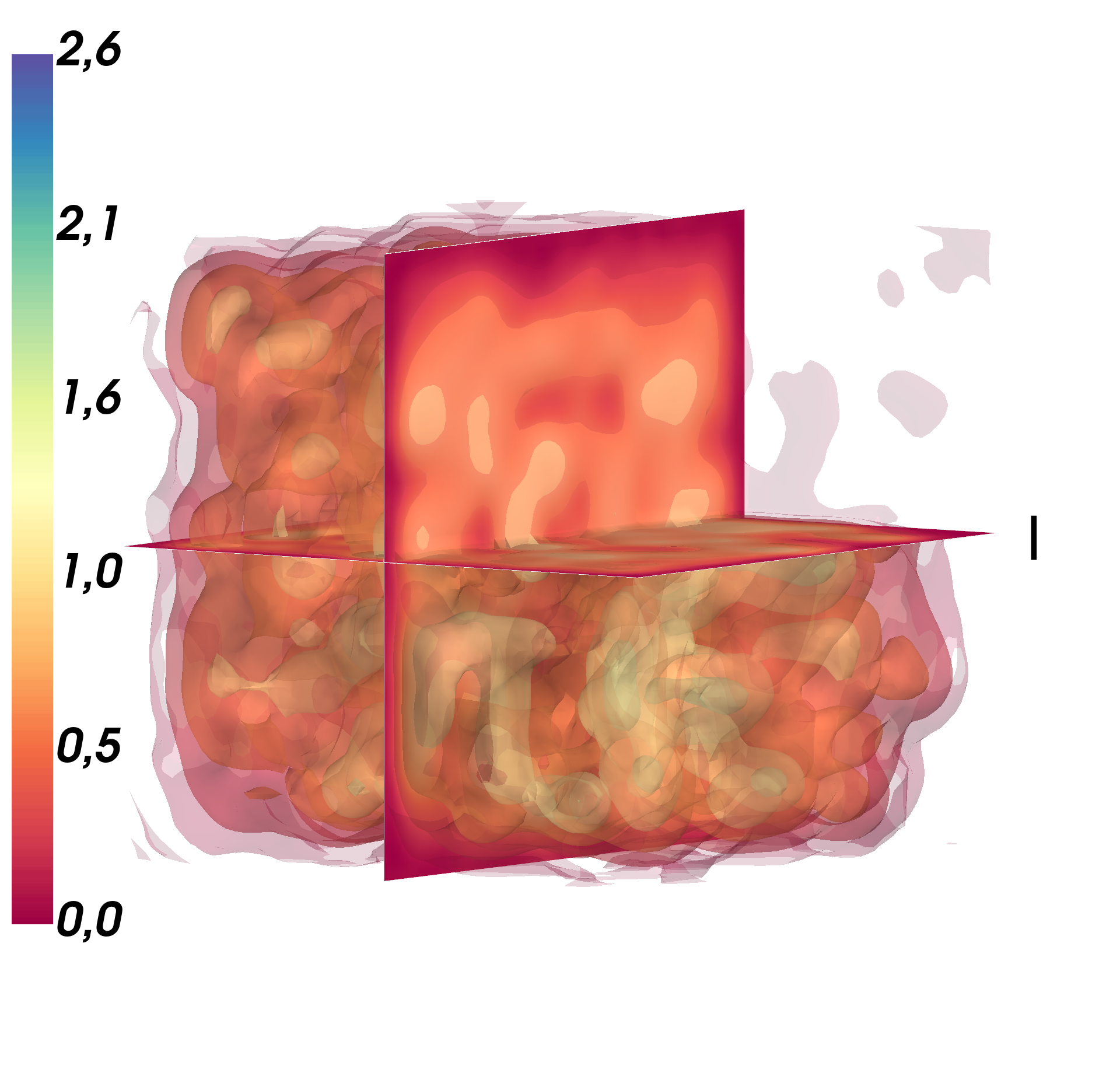}
            \caption{$N=95000$}
            \label{fig:ReconLinearRandomLargeN}
        \end{subfigure}
        
        \caption{The left panel shows the exact source strength with randomly chosen linear-spline coefficients $c^{\lambda}_{k,l,m}\sim |X^{\lambda}_{k,l,m}|$. The center and right panels show reconstructions for a fixed wave number $\kappa=12$, distance $R=4$, regularization penalty $\Vert \cdot\Vert_{L^2}$, and different sample sizes.}
        \label{fig:CompareRandomSampleNumberComparison}
\end{figure}

\input{TableOfRatesNew}
\input{PlotsBySmooth_new}
\Cref{fig:ErrorPlots} plots the error against the data-noise level in four subplots, one for each exact solution. 
As expected, the error in the stronger $H^1$-norm is larger and converges more slowly than in the $L^2$-norm for the same setup. It is clear that larger wave numbers $\kappa$ lead to steeper curves, i.e., faster convergence. 

The reconstruction errors are displayed on log--log versus log scales, so that the asymptotic slopes of the lines correspond to the exponent $p$ in the logarithmic convergence rate $C\log(3+\delta^{-2})^p$ from \cref{thm:ConvergenceRate}. 
\Cref{tab:RateLines} displays the slopes of the best-fitting straight lines in \cref{fig:ErrorPlots} as estimates of the order of convergence. 
Our theoretical upper bounds on the logarithmic convergence exponent $p$ in \cref{thm:ConvergenceRate} are shown in the last column of \cref{tab:RateLines}. 
Note the increase of this estimated order of convergence with increasing smoothness. 

We observe that the theoretical logarithmic convergence orders are upper bounds of the numerically observed 
convergence orders, but they are larger by a small factor. This may indicate that the exponents in our 
upper bounds can be improved by a small factor, and indeed we do not have any proof or heuristic argument 
of optimality. 
However, we do not consider our numerical results strong evidence that our upper bounds are suboptimal,
because an accurate numerical estimate of the optimal exponent in logarithmic stability estimates  
is challenging for several reasons. To capture the true asymptotic behavior, the data-noise level must be very small. Moreover, to ensure that other sources of error, such as discretization errors, are not dominant, very fine discretizations are required, and these are limited by the available computational resources. 
(For example, we do not know whether the slowdown of the convergence rate observed in some of the curves in \cref{fig:ErrorPlots} occurs for all curves at smaller noise levels.) 
Finally, the rates are defined 
by a supremum over noise realizations, which is difficult to approximate numerically. 

\paragraph{Wave number dependence for flat sources}
Solving the full 3D problem for wave numbers larger than $\kappa = 12$ was computationally infeasible 
given the computational resources required for fine discretizations of both the domain and the codomain. 
To investigate the dependence of the limiting behavior on the wave number, we consider a setup in which the random sources are supported only on one slice of the cube $D$. The sources can therefore be regarded as being supported in a thin layer, allowing us to discretize $f\colon D \to [0,\infty)$ on a uniform two-dimensional square grid of size $128 \times 128$. In addition, the noise model is simplified to additive Gaussian noise, which yields small noise levels without requiring an excessively large number of samples. To avoid an inverse crime, the ground truth was generated on a finer uniform grid and used to construct the exact data $\COp^\dagger$. In this setting, we can consider wave numbers in $\{7,14,28,42\}$ by adapting the degree of the spherical harmonics.

\input{SurfacePlot}
The reconstructions in \cref{fig:resultionKappa} illustrate particularly clearly how the resolution improves as the wave number increases. The convergence rates are displayed in \cref{fig:RatesSurface}, where the limiting behavior appears similar for all wave numbers. However, for higher wave numbers, the onset of this limiting behavior occurs at smaller noise levels.
\input{SurfaceRate}

%% file: TableOfRatesNew.tex
\begin{small}
\begin{table}[ht!]
    \centering
	\caption{estimated coefficients for relative error by line fitting with ansatz $e^c\log(3+\delta^{-2})^p$ and corresponding expected order $p$ of the convergence rates from \cref{thm:ConvergenceRate}}
	\label{tab:RateLines}
\begin{tabular}{|c|c||c|c|c||c||}
\hline
	\multicolumn{2}{|c||}{} & &&& rates from \\
	\multicolumn{2}{|c||}{} & $\kappa=6$ & $\kappa=9$ & $\kappa=12$ & \Cref{thm:ConvergenceRate} \\\hline
	 $X$ & $q^\dagger$ & $p$ & $p$ & $p$ & $p$\\ \hline
	\multirow{4}{*}{$H^1$} 
    & cubic spl.\ as in \cref{fig:ExactSolCubic} & -4.7817 & -6.1427 & -6.4810 & -2.5 \\\cline{2-5}
	& linear spl.\ as in \cref{fig:ExactSolLinear} & -1.3528 & -2.9504 & -3.0922 & -0.5 \\\cline{2-5}
	  & random cubic spl.         & -3.4636 & -3.8854 & -3.8625 & -2.5 \\\cline{2-5}
	  & random linear spl.        & -1.5247 & -2.1500 & -2.1333 & -0.5 \\\cline{2-5}
	\hline\hline
	\multirow{4}{*}{$L^2$}  
    & cubic spl.\ as in \cref{fig:ExactSolCubic} & -5.8884 & -6.6516 & -6.4471 & -3.5 \\\cline{2-5}
	& linear spl.\ as in \cref{fig:ExactSolLinear} & -2.4377 & -5.2667 & -5.4174 & -1.5 \\\cline{2-5}
	& random cubic spl.         & -4.6238 & -4.9408 & -5.1274 & -3.5 \\\cline{2-5}
	& random linear spl.        & -2.6023 & -3.6699 & -3.5808 & -1.5 \\\cline{2-5}
	\hline
\end{tabular}
\end{table}
\end{small}

%% file: PlotsBySmooth_new.tex
\begin{figure}

\centering
	\begin{tikzpicture}[]
		\begin{loglogaxis}[
        title={$q^{\dagger}$ as in \cref{fig:ExactSolLinear}},
        title style={font = \small,above right,at={(0,0)},draw=black,fill=white},
        width = 0.48\textwidth,
        height= 0.25\textheight,
        yshift = 0.00\textheight,
        xshift = 0.00\textwidth,
        xmin=3.40,
        xmax=11.60,
        ymin=0.02600,
        ymax=6.20000,
        xtick = {3.9143946581,6.2169797511,8.7498233534,11.0524084464},
        xticklabels={$10^{-1.7}$,$10^{-2.7}$,$10^{-3.8}$,$10^{-4.8}$
        },
        ylabel = {$\Vert\hat{q}_\alpha-q^\dagger\Vert_X$},
        ticklabel style={font=\footnotesize},
		label style={font=\footnotesize},
        ]

        \addplot [solid,color={rgb,1:red,0.001462;green,0.000466;blue,0.013866},mark=asterisk,mark options={scale =0.5,solid}] coordinates{
			 (6.2324689627,0.7725610808)
			 (6.5577517585,0.6100551749)
			 (6.7475104840,0.5648863855)
			 (7.0252033232,0.4236998829)
			 (7.1848100950,0.4226471105)
			 (7.4215544341,0.3388607774)
			 (7.6558063374,0.3290445209)
			 (7.8376787186,0.2738381691)
			 (8.0618348615,0.2307755567)
			 (8.1958337534,0.2006318350)
			 (8.5487912005,0.1772232540)
			 (8.7334565818,0.1531210563)
			 (8.9825563385,0.1361842861)
			 (9.1513621120,0.1215046249)
			 (9.3998560596,0.1081687257)
			 (9.6038960344,0.0995462383)
			 (9.8136856534,0.0919034492)
			 (10.0150981680,0.0866288122)
			};

        \addplot [solid,color={rgb,1:red,0.001462;green,0.000466;blue,0.013866},mark=oplus,mark options={scale =0.5,solid}] coordinates{
			 (6.2331168167,0.8975074728)
			 (6.5970474023,0.8002659547)
			 (6.7895366655,0.6606468359)
			 (7.0386001604,0.5745720953)
			 (7.2156990467,0.4849313416)
			 (7.4372489982,0.4414235054)
			 (7.6661956778,0.3419645179)
			 (7.8469135154,0.2682950730)
			 (8.0526374924,0.2757003095)
			 (8.1995452128,0.2225158971)
			 (8.5419705963,0.1589661705)
			 (8.7418421703,0.1793237761)
			 (8.9931134499,0.1305898073)
			 (9.1603095372,0.1080171681)
			 (9.3761968735,0.0878326778)
			 (9.5892367147,0.0748722787)
			 (9.7780638270,0.0652768306)
			 (10.0063473582,0.0551479059)
			};

        \addplot [solid,color={rgb,1:red,0.001462;green,0.000466;blue,0.013866},mark=triangle,mark options={scale =0.5,solid}] coordinates{
			 (6.2092055840,1.0669245964)
			 (6.5710059688,0.8746932987)
			 (6.7660797560,0.7814210386)
			 (7.0253502838,0.6168618753)
			 (7.2012314879,0.5825295304)
			 (7.4201385643,0.4660156448)
			 (7.6432810555,0.4227372016)
			 (7.8352219521,0.3407653453)
			 (8.0453944354,0.3222155344)
			 (8.1776954516,0.2771190099)
			 (8.5309371653,0.1758470036)
			 (8.7310334133,0.1334451356)
			 (8.9906561728,0.1448969114)
			 (9.1612005546,0.1167692265)
			 (9.3867943202,0.0878630932)
			 (9.6061507215,0.0737196838)
			 (9.7994572423,0.0621671519)
			 (10.0063359216,0.0581526107)
			};

        \addplot [solid,color={rgb,1:red,0.987622;green,0.64532;blue,0.039886},mark=asterisk,mark options={scale =0.5,solid}] coordinates{
			 (6.2324689627,0.5897334801)
			 (6.5577517585,0.4248481882)
			 (6.7475104840,0.4164718984)
			 (7.0252033232,0.3009355572)
			 (7.1848100950,0.2554870261)
			 (7.4215544341,0.2653652129)
			 (7.6558063374,0.2014523832)
			 (7.8376787186,0.1637851147)
			 (8.0618348615,0.1355740193)
			 (8.1958337534,0.1171072320)
			 (8.5487912005,0.0905014897)
			 (8.7334565818,0.0775146551)
			 (8.9825563385,0.0687185729)
			 (9.1513621120,0.0589623640)
			 (9.3998560596,0.0592936520)
			 (9.6038960344,0.0504020816)
			 (9.8136856534,0.0431238819)
			 (10.0150981680,0.0387616263)
			};

        \addplot [solid,color={rgb,1:red,0.987622;green,0.64532;blue,0.039886},mark=oplus,mark options={scale =0.5,solid}] coordinates{
			 (6.2331168167,0.7051265742)
			 (6.5970474023,0.5599484283)
			 (6.7895366655,0.4508867838)
			 (7.0386001604,0.4094335997)
			 (7.2156990467,0.3305188199)
			 (7.4372489982,0.2543294121)
			 (7.6661956778,0.2697664483)
			 (7.8469135154,0.2078408535)
			 (8.0526374924,0.1608582373)
			 (8.1995452128,0.1334270183)
			 (8.5419705963,0.1017061502)
			 (8.7418421703,0.0837586794)
			 (8.9931134499,0.0699535665)
			 (9.1603095372,0.0599106579)
			 (9.3761968735,0.0526185307)
			 (9.5892367147,0.0528605215)
			 (9.7780638270,0.0430820582)
			 (10.0063473582,0.0326294478)
			};

        \addplot [solid,color={rgb,1:red,0.987622;green,0.64532;blue,0.039886},mark=triangle,mark options={scale =0.5,solid}] coordinates{
			 (6.2092055840,0.7842945470)
			 (6.5710059688,0.6299054414)
			 (6.7660797560,0.5346716782)
			 (7.0253502838,0.4886692284)
			 (7.2012314879,0.4054827162)
			 (7.4201385643,0.3048621790)
			 (7.6432810555,0.2266249628)
			 (7.8352219521,0.2552848842)
			 (8.0453944354,0.1900458614)
			 (8.1776954516,0.1593139895)
			 (8.5309371653,0.1130116545)
			 (8.7310334133,0.0952670246)
			 (8.9906561728,0.0818555068)
			 (9.1612005546,0.0726230647)
			 (9.3867943202,0.0627184096)
			 (9.6061507215,0.0626430074)
			 (9.7994572423,0.0511414013)
			 (10.0063359216,0.0405498549)
			};

		\addplot [solid,color={rgb,1:red,0.987622;green,0.64532;blue,0.039886}] coordinates{
			(3.6841361488,9.2008954613)
			(4.5540016284,4.3849377059)
			(5.4238671079,2.3785111908)
			(6.2937325875,1.4132362734)
			(7.1635980671,0.8983293462)
			(8.0334635467,0.6014992493)
			(8.9033290262,0.4197218024)
			(9.7731945058,0.3028775732)
			(10.6430599854,0.2247306099)
			(11.5129254650,0.1707045470)
			} % p=3.5
			node[pos=0.8, sloped,above,text={rgb,1:red,0.987622;green,0.64532;blue,0.039886}] {\footnotesize $p=-3.5$};
		\addplot [solid,color={rgb,1:red,0.001462;green,0.000466;blue,0.013866}] coordinates{
			(3.6841361488,6.1030892516)
			(4.5540016284,3.5945535508)
			(5.4238671079,2.3221436400)
			(6.2937325875,1.6010169167)
			(7.1635980671,1.1583488117)
			(8.0334635467,0.8697820297)
			(8.9033290262,0.6726458373)
			(9.7731945058,0.5328146584)
			(10.6430599854,0.4305278452)
			(11.5129254650,0.3537555708)
			}% p=2.5
			node[pos=0.8, sloped,above,text={rgb,1:red,0.001462;green,0.000466;blue,0.013866}] {\footnotesize $p=-2.5$};
		
		\end{loglogaxis}
		\begin{loglogaxis}[
        title={$q^{\dagger}$ as in \cref{fig:ExactSolCubic}},
        title style={font = \small,above right,at={(0,0)},draw=black,fill=white},
        width = 0.48\textwidth,
        height= 0.25\textheight,
        yshift = 0.00\textheight,
        xshift = 0.48\textwidth,
        xmin=3.40,
        xmax=11.60,
        ymin=0.02600,
        ymax=6.20000,
        xtick = {3.9143946581,6.2169797511,8.7498233534,11.0524084464},
        xticklabels={$10^{-1.7}$,$10^{-2.7}$,$10^{-3.8}$,$10^{-4.8}$
        },
		label style={font=\footnotesize},
        ]

        \addplot [solid,color={rgb,1:red,0.001462;green,0.000466;blue,0.013866},mark=asterisk,mark options={scale =0.5,solid}] coordinates{
			 (5.7227536296,1.9179428148)
			 (6.0551662832,1.6814574411)
			 (6.2536075939,1.5788853763)
			 (6.5067013171,1.4508460503)
			 (6.6861419321,1.3746940729)
			 (6.9335671446,1.2753152385)
			 (7.1419725755,1.2330163204)
			 (7.3363485867,1.1558833863)
			 (7.5575688321,1.1312823804)
			 (7.6568575360,1.1056439399)
			 (8.0096301771,1.0759240430)
			 (8.2072808029,1.0435416885)
			 (8.4797102068,1.0035759594)
			 (8.6590893450,0.9854539606)
			 (8.8799488127,0.9669974239)
			 (9.0951949043,0.9683944864)
			 (9.2948316676,0.9530818679)
			 (9.4924778427,0.9361900906)
			};

        \addplot [solid,color={rgb,1:red,0.001462;green,0.000466;blue,0.013866},mark=oplus,mark options={scale =0.5,solid}] coordinates{
			 (5.7196367347,2.1125522389)
			 (6.0605291105,1.8344872937)
			 (6.2737497971,1.7024984682)
			 (6.5277188130,1.5218796103)
			 (6.6953115574,1.4238825828)
			 (6.9151361913,1.2598544037)
			 (7.1268643318,1.1595730549)
			 (7.3126264800,1.0449539766)
			 (7.5241432163,0.9607219018)
			 (7.6627727219,0.8551398247)
			 (8.0043339112,0.7374713292)
			 (8.2106273207,0.7174744330)
			 (8.4729601708,0.6209635868)
			 (8.6402103218,0.6320102300)
			 (8.8710862516,0.5793901021)
			 (9.0851182369,0.5574951868)
			 (9.2738675772,0.5473230496)
			 (9.4825130851,0.5491937159)
			};

        \addplot [solid,color={rgb,1:red,0.001462;green,0.000466;blue,0.013866},mark=triangle,mark options={scale =0.5,solid}] coordinates{
			 (5.7028398499,2.2763935532)
			 (6.0543220129,1.9679154228)
			 (6.2545538645,1.8127673203)
			 (6.5065704074,1.6348222931)
			 (6.6789773149,1.5414685646)
			 (6.9070652077,1.3674287698)
			 (7.1270441012,1.2715212732)
			 (7.3204109587,1.1705833469)
			 (7.5259313265,1.0549311784)
			 (7.6635189427,0.9540543227)
			 (8.0104057277,0.8047261478)
			 (8.2150274105,0.7476722205)
			 (8.4699018069,0.7051102645)
			 (8.6351533529,0.6354180369)
			 (8.8633763700,0.5820164814)
			 (9.0821174645,0.5744721663)
			 (9.2710062217,0.5462876459)
			 (9.4859505435,0.5301083650)
			};

        \addplot [solid,color={rgb,1:red,0.987622;green,0.64532;blue,0.039886},mark=asterisk,mark options={scale =0.5,solid}] coordinates{
			 (5.7227536296,1.1607190626)
			 (6.0551662832,0.9826752405)
			 (6.2536075939,0.8607206707)
			 (6.5067013171,0.6996378283)
			 (6.6861419321,0.6023499293)
			 (6.9335671446,0.5327339646)
			 (7.1419725755,0.5179638708)
			 (7.3363485867,0.4552657848)
			 (7.5575688321,0.4160228509)
			 (7.6568575360,0.4396592939)
			 (8.0096301771,0.3882839714)
			 (8.2072808029,0.3731151024)
			 (8.4797102068,0.3548962560)
			 (8.6590893450,0.3641872968)
			 (8.8799488127,0.3493205283)
			 (9.0951949043,0.3405533186)
			 (9.2948316676,0.3333951258)
			 (9.4924778427,0.3278130532)
			};

        \addplot [solid,color={rgb,1:red,0.987622;green,0.64532;blue,0.039886},mark=oplus,mark options={scale =0.5,solid}] coordinates{
			 (5.7196367347,1.3377203425)
			 (6.0605291105,1.0880157176)
			 (6.2737497971,0.9551290648)
			 (6.5277188130,0.7861728215)
			 (6.6953115574,0.6871606180)
			 (6.9151361913,0.5865707404)
			 (7.1268643318,0.5235261819)
			 (7.3126264800,0.3946080241)
			 (7.5241432163,0.3725566561)
			 (7.6627727219,0.2946991351)
			 (8.0043339112,0.2620684359)
			 (8.2106273207,0.2026892646)
			 (8.4729601708,0.1614026465)
			 (8.6402103218,0.1444931846)
			 (8.8710862516,0.1545644455)
			 (9.0851182369,0.1328344455)
			 (9.2738675772,0.1228657034)
			 (9.4825130851,0.1171838534)
			};

        \addplot [solid,color={rgb,1:red,0.987622;green,0.64532;blue,0.039886},mark=triangle,mark options={scale =0.5,solid}] coordinates{
			 (5.7028398499,1.5147888091)
			 (6.0543220129,1.1973028373)
			 (6.2545538645,1.0280928605)
			 (6.5065704074,0.9006598498)
			 (6.6789773149,0.7694207601)
			 (6.9070652077,0.6668045623)
			 (7.1270441012,0.5729790850)
			 (7.3204109587,0.4384454712)
			 (7.5259313265,0.3792523448)
			 (7.6635189427,0.4050544035)
			 (8.0104057277,0.2461613105)
			 (8.2150274105,0.2610930224)
			 (8.4699018069,0.1886158917)
			 (8.6351533529,0.1599769322)
			 (8.8633763700,0.1370696097)
			 (9.0821174645,0.1476421300)
			 (9.2710062217,0.1266052590)
			 (9.4859505435,0.1141392530)
			};

		\addplot [solid,color={rgb,1:red,0.987622;green,0.64532;blue,0.039886}] coordinates{
			(3.6841361488,4.2481833526)
			(4.5540016284,3.0921488715)
			(5.4238671079,2.3790680696)
			(6.2937325875,1.9033163246)
			(7.1635980671,1.5673898662)
			(8.0334635467,1.3198351137)
			(8.9033290262,1.1312153649)
			(9.7731945058,0.9836013557)
			(10.6430599854,0.8655141399)
			(11.5129254650,0.7692993981)
			}
			node[pos=0.8, sloped,above,text={rgb,1:red,0.987622;green,0.64532;blue,0.039886}] {\footnotesize $p=-1.5$};% p=1.5
		\addplot [solid,color={rgb,1:red,0.001462;green,0.000466;blue,0.013866}] coordinates{
			(3.914395,2.858970)
			(4.758676,2.593146)
			(5.602957,2.389822)
			(6.447238,2.227861)
			(7.291519,2.094913)
			(8.135801,1.983238)
			(8.980082,1.887708)
			(9.824363,1.804774)
			(10.668644,1.731890)
			(11.512925,1.667179)
			}
			node[pos=0.8, sloped,above,text={rgb,1:red,0.001462;green,0.000466;blue,0.013866}] {\footnotesize $p=-0.5$};%p=0.5
		
		\end{loglogaxis}
		\begin{loglogaxis}[
        title={$q^{\dagger}$ random cubic spl.},
        title style={font = \small,above right,at={(0,0)},draw=black,fill=white},
        width = 0.48\textwidth,
        height= 0.25\textheight,
        yshift = -0.20\textheight,
        xshift = 0.00\textwidth,
        xmin=3.40,
        xmax=11.60,
        ymin=0.02600,
        ymax=6.20000,
        xtick = {3.9143946581,6.2169797511,8.7498233534,11.0524084464},
        xticklabels={$10^{-1.7}$,$10^{-2.7}$,$10^{-3.8}$,$10^{-4.8}$
        },
        xlabel = {$\delta=\Vert \COp^{\obs}-\COp^\dagger\Vert_{\HS}$},
        ylabel = {$\Vert\hat{q}_\alpha-q^\dagger\Vert_X$},
        ticklabel style={font=\footnotesize},
		label style={font=\footnotesize},
        ]

        \addplot [solid,color={rgb,1:red,0.001462;green,0.000466;blue,0.013866},mark=asterisk,mark options={scale =0.5,solid}] coordinates{
			 (4.5025152231,2.1890655742)
			 (4.8483991030,1.8300491256)
			 (5.0506505716,1.5220638011)
			 (5.3096216625,1.3306924678)
			 (5.4820026531,1.3138845164)
			 (5.7230885617,1.1699386526)
			 (5.9306630502,1.0012488779)
			 (6.1245052105,0.9102601074)
			 (6.3235734662,0.8659887804)
			 (6.4599285771,0.7821815350)
			 (6.8094977113,0.6661296355)
			 (7.0099167072,0.5564942591)
			 (7.2728563927,0.4500913681)
			 (7.4322817291,0.4663149318)
			 (7.6639564098,0.3865016804)
			 (7.8862701179,0.3116725842)
			 (8.0803150955,0.2719198465)
			 (8.2902076560,0.2930676748)
			};

        \addplot [solid,color={rgb,1:red,0.001462;green,0.000466;blue,0.013866},mark=oplus,mark options={scale =0.5,solid}] coordinates{
			 (4.5176066996,2.6627135527)
			 (4.8611171754,2.2604826672)
			 (5.0626772170,2.0725163382)
			 (5.3160945646,1.7191703834)
			 (5.4837288193,1.5431606646)
			 (5.7075150316,1.4310946482)
			 (5.9280327717,1.2323715894)
			 (6.1180733725,1.1384820106)
			 (6.3283139075,0.9549006337)
			 (6.4644992610,0.9230758952)
			 (6.8111387217,0.7606509454)
			 (7.0140328893,0.6268246250)
			 (7.2724499494,0.5889054794)
			 (7.4385418771,0.4622952305)
			 (7.6652140540,0.3810403545)
			 (7.8839800047,0.3767257383)
			 (8.0761982822,0.3116165430)
			 (8.2898884693,0.2368544509)
			};

        \addplot [solid,color={rgb,1:red,0.001462;green,0.000466;blue,0.013866},mark=triangle,mark options={scale =0.5,solid}] coordinates{
			 (4.5060641491,3.2036888161)
			 (4.8559981811,2.5707816126)
			 (5.0577262865,2.2968842502)
			 (5.3153682418,2.0819117610)
			 (5.4844143909,1.9533271157)
			 (5.7098444641,1.6288974086)
			 (5.9290826467,1.3739574699)
			 (6.1230807047,1.2918729780)
			 (6.3337111427,1.1616916550)
			 (6.4660257823,1.0697244126)
			 (6.8123642793,0.8206518979)
			 (7.0170473593,0.7851346501)
			 (7.2724365150,0.6675214908)
			 (7.4408836898,0.5451644987)
			 (7.6661721636,0.4841182128)
			 (7.8849571693,0.3894500129)
			 (8.0779008753,0.3350003553)
			 (8.2881354342,0.3226351929)
			};

        \addplot [solid,color={rgb,1:red,0.987622;green,0.64532;blue,0.039886},mark=asterisk,mark options={scale =0.5,solid}] coordinates{
			 (4.5025152231,1.8354033888)
			 (4.8483991030,1.5118144580)
			 (5.0506505716,1.2994316401)
			 (5.3096216625,0.9956339186)
			 (5.4820026531,0.9813425935)
			 (5.7230885617,0.8034700942)
			 (5.9306630502,0.7223447401)
			 (6.1245052105,0.6395740532)
			 (6.3235734662,0.4651597575)
			 (6.4599285771,0.4840580910)
			 (6.8094977113,0.4170337574)
			 (7.0099167072,0.3122819067)
			 (7.2728563927,0.2363974418)
			 (7.4322817291,0.2013470457)
			 (7.6639564098,0.2017196303)
			 (7.8862701179,0.1612983221)
			 (8.0803150955,0.1347622210)
			 (8.2902076560,0.1147975434)
			};

        \addplot [solid,color={rgb,1:red,0.987622;green,0.64532;blue,0.039886},mark=oplus,mark options={scale =0.5,solid}] coordinates{
			 (4.5176066996,2.3778314505)
			 (4.8611171754,1.9069798281)
			 (5.0626772170,1.6967061235)
			 (5.3160945646,1.3866231365)
			 (5.4837288193,1.2070399157)
			 (5.7075150316,1.1044701447)
			 (5.9280327717,0.8659212217)
			 (6.1180733725,0.7464002338)
			 (6.3283139075,0.6719431269)
			 (6.4644992610,0.6211484268)
			 (6.8111387217,0.4530111946)
			 (7.0140328893,0.3303932956)
			 (7.2724499494,0.3317060822)
			 (7.4385418771,0.2523085654)
			 (7.6652140540,0.1937038403)
			 (7.8839800047,0.2084388546)
			 (8.0761982822,0.1560966959)
			 (8.2898884693,0.1166915099)
			};

        \addplot [solid,color={rgb,1:red,0.987622;green,0.64532;blue,0.039886},mark=triangle,mark options={scale =0.5,solid}] coordinates{
			 (4.5060641491,2.9354834148)
			 (4.8559981811,2.2839331811)
			 (5.0577262865,2.0045532908)
			 (5.3153682418,1.7041631288)
			 (5.4844143909,1.5918357098)
			 (5.7098444641,1.2984356559)
			 (5.9290826467,1.0169245011)
			 (6.1230807047,0.9367501319)
			 (6.3337111427,0.7761215932)
			 (6.4660257823,0.7525971299)
			 (6.8123642793,0.4965414502)
			 (7.0170473593,0.4752751513)
			 (7.2724365150,0.3148512160)
			 (7.4408836898,0.3608306510)
			 (7.6661721636,0.2445247056)
			 (7.8849571693,0.1855736719)
			 (8.0779008753,0.1400182240)
			 (8.2881354342,0.1518733087)
			};

		\addplot [solid,color={rgb,1:red,0.987622;green,0.64532;blue,0.039886}] coordinates{
			(3.914395,8.784543)
			(4.758676,4.436499)
			(5.602957,2.505019)
			(6.447238,1.532739)
			(7.291519,0.996354)
			(8.135801,0.679010)
			(8.980082,0.480614)
			(9.824363,0.350923)
			(10.668644,0.262963)
			(11.512925,0.201431)
			}% p=3.5
			node[pos=0.8, sloped,above,text={rgb,1:red,0.987622;green,0.64532;blue,0.039886}] {\footnotesize $p=-3.5$};
		\addplot [solid,color={rgb,1:red,0.001462;green,0.000466;blue,0.013866}] coordinates{
			(3.914395,6.411958)
			(4.758676,3.936201)
			(5.602957,2.616800)
			(6.447238,1.842395)
			(7.291519,1.354480)
			(8.135801,1.029952)
			(8.980082,0.804668)
			(9.824363,0.642772)
			(10.668644,0.523052)
			(11.512925,0.432368)
			}% p=2.5
			node[pos=0.8, sloped,above,text={rgb,1:red,0.001462;green,0.000466;blue,0.013866}] {\footnotesize $p=-2.5$};
		
		\end{loglogaxis}
		\begin{loglogaxis}[
        title={$q^{\dagger}$ random linear spl.},
        title style={font = \small,above right,at={(0,0)},draw=black,fill=white},
        width = 0.48\textwidth,
        height= 0.25\textheight,
        yshift = -0.20\textheight,
        xshift = 0.48\textwidth,
        xmin=3.40,
        xmax=11.60,
        ymin=0.02600,
        ymax=6.20000,
        xtick = {3.9143946581,6.2169797511,8.7498233534,11.0524084464},
        xticklabels={$10^{-1.7}$,$10^{-2.7}$,$10^{-3.8}$,$10^{-4.8}$
        },
        xlabel = {$\delta=\Vert \COp^{\obs}-\COp^\dagger\Vert_{\HS}$},
        label style={font=\footnotesize},
        ]

        \addplot [solid,color={rgb,1:red,0.001462;green,0.000466;blue,0.013866},mark=asterisk,mark options={scale =0.5,solid}] coordinates{
			 (3.8687627495,4.3440429381)
			 (4.2143108795,4.0258604363)
			 (4.4137892978,3.8483503342)
			 (4.6701522997,3.7488756933)
			 (4.8374879084,3.7138865165)
			 (5.0683659859,3.4403916606)
			 (5.2830698761,3.2503597954)
			 (5.4792235119,3.0260005836)
			 (5.6969688693,2.7743352135)
			 (5.8271316661,2.6115218455)
			 (6.1752651765,2.3188133142)
			 (6.3750501555,2.2192133639)
			 (6.6313167725,2.1033385651)
			 (6.8020116004,1.9568276721)
			 (7.0176157940,1.9102184596)
			 (7.2361178854,1.8135388142)
			 (7.4341310986,1.7367916141)
			 (7.6410824461,1.7277266606)
			};

        \addplot [solid,color={rgb,1:red,0.001462;green,0.000466;blue,0.013866},mark=oplus,mark options={scale =0.5,solid}] coordinates{
			 (3.8634703566,5.1012938144)
			 (4.2159391798,4.4315474722)
			 (4.4184794466,4.1621718162)
			 (4.6760994750,3.8999547133)
			 (4.8407370737,3.8070555259)
			 (5.0638627383,3.7285732134)
			 (5.2820743794,3.5705842484)
			 (5.4777701483,3.2670423761)
			 (5.6887911213,3.0337284098)
			 (5.8207139245,2.8283673139)
			 (6.1695679633,2.3724717546)
			 (6.3754732853,2.1586902062)
			 (6.6334502424,1.9455314860)
			 (6.8009826160,1.6945767021)
			 (7.0261591009,1.5294017098)
			 (7.2428549596,1.4634501294)
			 (7.4340867914,1.2884303544)
			 (7.6418243606,1.2413809968)
			};

        \addplot [solid,color={rgb,1:red,0.001462;green,0.000466;blue,0.013866},mark=triangle,mark options={scale =0.5,solid}] coordinates{
			 (3.8665445464,5.7489343457)
			 (4.2140056037,4.9421478353)
			 (4.4192938481,4.5389610632)
			 (4.6734315727,4.1466405868)
			 (4.8424446214,4.0484346773)
			 (5.0680532086,3.8799263015)
			 (5.2867342292,3.7741308289)
			 (5.4803702440,3.7090615086)
			 (5.6901671445,3.3510025218)
			 (5.8267755059,3.2369506447)
			 (6.1739370069,2.7563937710)
			 (6.3758887794,2.4656758831)
			 (6.6324537471,2.0664785084)
			 (6.7991247403,1.9256336025)
			 (7.0238022293,1.7356338609)
			 (7.2434229833,1.6274468411)
			 (7.4356497076,1.3923188298)
			 (7.6444203369,1.3138547505)
			};

        \addplot [solid,color={rgb,1:red,0.987622;green,0.64532;blue,0.039886},mark=asterisk,mark options={scale =0.5,solid}] coordinates{
			 (3.8687627495,3.1209272972)
			 (4.2143108795,2.6680614639)
			 (4.4137892978,2.4120097530)
			 (4.6701522997,2.1589067097)
			 (4.8374879084,2.0934380293)
			 (5.0683659859,1.7803246547)
			 (5.2830698761,1.6844096705)
			 (5.4792235119,1.5513390065)
			 (5.6969688693,1.3024863963)
			 (5.8271316661,1.2829605834)
			 (6.1752651765,1.0507680561)
			 (6.3750501555,0.9183203074)
			 (6.6313167725,0.8372603094)
			 (6.8020116004,0.7619794947)
			 (7.0176157940,0.6708444564)
			 (7.2361178854,0.6862343790)
			 (7.4341310986,0.6282296332)
			 (7.6410824461,0.5973273578)
			};

        \addplot [solid,color={rgb,1:red,0.987622;green,0.64532;blue,0.039886},mark=oplus,mark options={scale =0.5,solid}] coordinates{
			 (3.8634703566,4.0566694526)
			 (4.2159391798,3.2118001536)
			 (4.4184794466,2.8089939808)
			 (4.6760994750,2.4720432223)
			 (4.8407370737,2.3357048053)
			 (5.0638627383,2.1908252206)
			 (5.2820743794,2.0699895995)
			 (5.4777701483,1.6597490861)
			 (5.6887911213,1.5617722167)
			 (5.8207139245,1.3324889107)
			 (6.1695679633,1.0017062431)
			 (6.3754732853,0.8835568915)
			 (6.6334502424,0.7719819224)
			 (6.8009826160,0.6130467475)
			 (7.0261591009,0.5720604946)
			 (7.2428549596,0.4340541372)
			 (7.4340867914,0.3545162408)
			 (7.6418243606,0.3737746740)
			};

        \addplot [solid,color={rgb,1:red,0.987622;green,0.64532;blue,0.039886},mark=triangle,mark options={scale =0.5,solid}] coordinates{
			 (3.8665445464,4.8318025607)
			 (4.2140056037,3.8831746718)
			 (4.4192938481,3.3190683719)
			 (4.6734315727,2.8786049313)
			 (4.8424446214,2.5641139161)
			 (5.0680532086,2.3486516749)
			 (5.2867342292,2.1966855205)
			 (5.4803702440,2.0814380392)
			 (5.6901671445,1.8150753704)
			 (5.8267755059,1.6309966927)
			 (6.1739370069,1.3465611341)
			 (6.3758887794,1.1074820826)
			 (6.6324537471,0.8814372412)
			 (6.7991247403,0.8665144186)
			 (7.0238022293,0.6682323833)
			 (7.2434229833,0.5894080085)
			 (7.4356497076,0.4733910105)
			 (7.6444203369,0.3544442658)
			};

		\addplot [solid,color={rgb,1:red,0.987622;green,0.64532;blue,0.039886}] coordinates{
			(3.6841361488,4.2481833526)
			(4.5540016284,3.0921488715)
			(5.4238671079,2.3790680696)
			(6.2937325875,1.9033163246)
			(7.1635980671,1.5673898662)
			(8.0334635467,1.3198351137)
			(8.9033290262,1.1312153649)
			(9.7731945058,0.9836013557)
			(10.6430599854,0.8655141399)
			(11.5129254650,0.7692993981)
			}
			node[pos=0.8, sloped,above,text={rgb,1:red,0.987622;green,0.64532;blue,0.039886}] {\footnotesize $p=-1.5$};% p=1.5
		
		\addplot [solid,color={rgb,1:red,0.001462;green,0.000466;blue,0.013866}] coordinates{
			(3.6841361488,5.1569091211)
			(4.5540016284,4.6388307290)
			(5.4238671079,4.2506696079)
			(6.2937325875,3.9460148741)
			(7.1635980671,3.6986855047)
			(8.0334635467,3.4927026999)
			(8.9033290262,3.3176978059)
			(9.7731945058,3.1666113689)
			(10.6430599854,3.0344485821)
			(11.5129254650,2.9175626549)
			}
			node[pos=0.8, sloped,above,text={rgb,1:red,0.001462;green,0.000466;blue,0.013866}] {\footnotesize $p=-0.5$};%p=0.5

		\end{loglogaxis}
        \begin{axis}[
			 xshift = 0\textwidth,
			 yshift = -0.38\textheight,
			 width=0.25\textwidth,
			 height=0.2\textheight,
			 hide axis,
			 xmin=10,
			 xmax=11,
			 ymin=0,
			 ymax=0.4,
			 legend columns=1,
			 legend style={font=\scriptsize, text width=0.25\textwidth, minimum height=1.5\baselineskip, below right, at={(0,1)}, inner xsep=1pt, style={column sep=2pt},legend cell align=center}
			]
			\addlegendimage{color={rgb,1:red,0.001462;green,0.000466;blue,0.013866},line width=5 pt}
			\addlegendentry{$X=H^1$};
			\addlegendimage{color={rgb,1:red,0.987622;green,0.64532;blue,0.039886},line width=5 pt}
			\addlegendentry{$X=L^2$};
		\end{axis}
        \begin{axis}[
			 xshift = 0.5\textwidth,
			 yshift = -0.38\textheight,
			 width=0.25\textwidth,
			 height=0.2\textheight,
			 hide axis,
			 xmin=10,
			 xmax=11,
			 ymin=0,
			 ymax=0.4,
			 legend columns=1,
			 legend style={font=\scriptsize, text width=0.25\textwidth, minimum height=1.5\baselineskip, below right, at={(0,1)}, inner xsep=10pt, style={column sep=2pt},legend cell align=center}
			]
			\addlegendimage{color=black, only marks, mark=o}
			\addlegendentry{$\kappa=6$};
			\addlegendimage{color=black, only marks, mark=asterisk}
			\addlegendentry{$\kappa=9$};
			\addlegendimage{color=black, only marks, mark=oplus}
			\addlegendentry{$\kappa=12$};
		\end{axis}
		\end{tikzpicture}
\caption{Each of the four panels shows the reconstruction error on a logarithmic scale over the data noise level $\delta$ on a double logarithmic scale 
for one kind of exact solution. For each line synthetic 
data were generated for 18 values of the sample size $N$ between 550 and 92500. 
The measurement radius $R$, the wave number $\kappa$, and the regularization norm are indicated 
by line type, marker type, and color, respectively. }
\label{fig:ErrorPlots}
\end{figure}

%% file: SurfacePlot.tex
\begin{figure}[thb]
    \centering
    \includegraphics{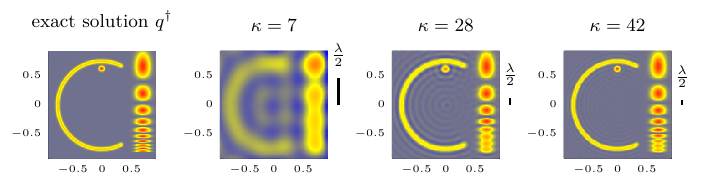}
    \caption{The reconstruction of a flat ground truth given by a unfinished circle with a dot very close at the bottom and bubbles of decreasing size on the left for different wave number $\kappa$}
    \label{fig:resultionKappa}
    \end{figure}    

%% file: SurfaceRate.tex
\begin{figure}[!hbt]
    \centering
    \begin{tikzpicture}[]
        \begin{loglogaxis}[
        width = 0.6\textwidth,
        height= 0.22\textheight,
        yshift = 0.00\textheight,
        xshift = 0.00\textwidth,
        xmin=4.9,
        xmax=9.3,
        ymin=1.5,
        ymax=11,
        ytick distance=10^(0.2),
        xtick = {5.0656872 , 5.98672124, 6.90775528, 7.82878932, 8.74982335},
        xticklabels={$10^{-2.2}$,$10^{-2.6}$,$10^{-3}$,$10^{-3.4}$,$10^{-3.8}$
        },
        ticklabel style={font=\footnotesize},
        label style={font=\footnotesize},
        xlabel = {$\delta=\Vert\COp^{\obs}-\COp^\dagger\Vert_{\HS}$},
        ylabel = {$\Vert\hat{q}_\alpha-q^\dagger\Vert_{L^2}$},
        legend pos=south west,
        ]
        \addplot+ coordinates{
            (5.1556560529, 7.8988940552)
			(5.3210353294, 7.5095433648)
			(5.5239764137, 7.1699106716)
			(5.7730105405, 7.0921807981)
			(6.0786066031, 6.7957702575)
			(6.4536112447, 6.5408986769)
			(6.9137889054, 6.3307204612)
			(7.4784845277, 6.1645788546)
			(8.1714367787, 5.9970070078)
			(9.0217759759, 5.9195410225)
        };
        \addlegendentry{$\kappa=7$}
        \addplot+ coordinates{
            (5.0739027776, 8.6013470703)
			(5.1614894913, 8.0956976691)
			(5.2823928127, 7.4781358132)
			(5.4492858551, 6.7878773760)
			(5.6796623848, 6.0792519385)
			(5.9976705212, 5.4341936878)
			(6.4366439488, 4.9444039949)
			(7.0425959607, 4.6340903015)
			(7.8790425832, 4.4711444435)
			(9.0336603390, 4.3505719221)
        };
        \addlegendentry{$\kappa=14$}
        \addplot+ coordinates{
            (5.3698554409, 9.5610344304)
			(5.3872555306, 9.3859037278)
			(5.4178044554, 9.1213291999)
			(5.4714384785, 8.7545286078)
			(5.5656024612, 8.2425510873)
			(5.7309239331, 7.5411224142)
			(6.0211749217, 5.8155174572)
			(6.5307617010, 4.0361150246)
			(7.4254311278, 3.1972894243)
			(8.9961810290, 2.6773706737)
        };
        \addlegendentry{$\kappa=28$}
        \addplot+ coordinates{
            (5.6931483902, 9.6322858225)
			(5.7103421545, 9.4880666315)
			(5.7403747723, 9.2769902214)
			(5.7928332172, 8.9546157194)
			(5.8844632062, 7.9487074472)
			(6.0445147459, 7.1517704335)
			(6.3240792821, 5.2258928593)
			(6.8123990448, 3.9730581334)
			(7.6653548378, 2.4181116557)
			(9.1552261073, 1.9369881569)
        };
        \addlegendentry{$\kappa=42$}
        \end{loglogaxis}
    \end{tikzpicture}
    \caption{Rates of the relative error in the $L^2$ norm for the surface source for different wave numbers.}
    \label{fig:RatesSurface}
\end{figure}

%% file: 06_Conclusion.tex
\section{Conclusion}\label{sec:conclusions}
In this work, we derived conditional stability estimates and convergence rates for spectral regularization methods applied to random inverse source problems under Sobolev-type smoothness assumptions on the unknown source strength.
For a fixed wave number, the resulting estimates are logarithmic in the noise level. Their explicit dependence on the wave number, however, yields Hölder-type convergence rates in suitable joint limits as the wave number tends to infinity while the noise level vanishes.

Numerical experiments corroborate the predicted asymptotically logarithmic convergence behavior with respect to the noise level. The observed exponents increase with the Sobolev smoothness index and appear to be slightly larger than the corresponding theoretical values.  
Moreover, in agreement with the theoretical analysis, the numerical results demonstrate that increasing the wave number shifts the onset of the asymptotic convergence regime to progressively smaller noise levels, leading to substantially improved reconstruction quality for practically relevant noise levels.

The present study opens several directions for future research, including a statistical convergence analysis that accounts for the Wishart distribution of correlation data with a specific covariance structure, stability results for the identification of unknown coefficients in partial differential equations from correlation data, and stability analyses for partial measurement data (e.g., open surfaces) or for particular classes of correlated sources.

\section*{Acknowledgments}
We would like to express our sincere gratitude to an anonymous referee for the thoughtful and constructive feedback. The detailed comments were highly valuable in improving the paper in several ways.
\bibliographystyle{siamplain}
\bibliography{refs}

%% file: refs.bib
@Book{GN:15,
  Title                    = {Mathematical foundations of infinite-dimensional statistical models},
  Author                   = {Gin{\'e}, Evarist and Nickl, Richard},
  Publisher                = {Cambridge University Press},
  Year                     = {2015},
  Volume                   = {40}
}

@Book{GT:77,
  Title                    = {Elliptic Partial Differential Equations of Second Order},
  Author                   = {David Gilbarg and Neil S Trudinger},
  Publisher                = {Springer},
  Year                     = {1977}
}

@Article{NHZ:20,
  author    = {Niu, Pingping and Helin, Tapio and Zhang, Zhidong},
  title     = {An inverse random source problem in a stochastic fractional diffusion equation},
  doi       = {10.1088/1361-6420/ab532c},
  issn      = {1361-6420},
  number    = {4},
  pages     = {045002},
  volume    = {36},
  journal   = {Inverse Problems},
  month     = feb,
  publisher = {IOP Publishing},
  year      = {2020},
}

@Article{TLSK:24,
  author    = {Triki, Faouzi and Linder-Steinlein, Kristoffer and Karamehmedovic, Mirza},
  title     = {Fourier method for inverse source problem using correlation of passive measurements},
  doi       = {10.1088/1361-6420/ad6fc7},
  issn      = {1361-6420},
  number    = {10},
  pages     = {105009},
  volume    = {40},
  journal   = {Inverse Problems},
  month     = sep,
  publisher = {IOP Publishing},
  year      = {2024},
}

@article{Artman:06,
  title = {Imaging passive seismic data},
  volume = {71},
  ISSN = {0016-8033},
  DOI = {10.1190/1.2209748},
  number = {4},
  journal = {Geophysics},
  publisher = {Society of Exploration Geophysicists},
  author = {Artman,  Brad},
  year = {2006},
  month = jan,
  pages = {SI177–SI187}
}

@Article{feizmohammadi:25,
  author    = {Feizmohammadi, Ali},
  title     = {Reconstruction of 1D Evolution Equations and their Initial Data from One Passive Measurement},
  doi       = {10.1137/25m1728533},
  issn      = {1095-7154},
  number    = {5},
  pages     = {5089–5106},
  volume    = {57},
  journal   = {SIAM Journal on Mathematical Analysis},
  month     = sep,
  publisher = {Society for Industrial & Applied Mathematics (SIAM)},
  year      = {2025},
}

@Article{BSS:08,
  author    = {Burov, V. A. and Sergeev, S. N. and Shurup, A. S.},
  title     = {The use of low-frequency noise in passive tomography of the ocean},
  doi       = {10.1134/s1063771008010077},
  issn      = {1562-6865},
  number    = {1},
  pages     = {42–51},
  volume    = {54},
  journal   = {Acoustical Physics},
  month     = jan,
  publisher = {Pleiades Publishing Ltd},
  year      = {2008},
}

@PhdThesis{Mickan:25,
  author      = {Mickan, Philipp Ronald},
  school = {University Goettingen},
  title       = {Conditional Stability and Instability Estimates for Inverse Random Source Problems},
  doi         = {10.53846/goediss-11736},
  year        = 2025
}

@Article{KRS:21,
  author     = {Koch, Herbert and R\"uland, Angkana and Salo, Mikko},
  title      = {On instability mechanisms for inverse problems},
  issn       = {2769-8505},
  pages      = {Paper No. 7, 93},
  fjournal   = {Ars Inveniendi Analytica},
  journal    = {Ars Inven. Anal.},
  mrclass    = {35R30 (49N45 65J22 65N21)},
  mrnumber   = {4462475},
  mrreviewer = {Yuchan\ Wang},
  year       = {2021},
  doi = {10.15781/5z5b3583 }
}

@article{Hohage2020,
author = {Thorsten Hohage and Hans-Georg Raumer and Carsten Spehr},
title = {Uniqueness of an inverse source problem in experimental aeroacoustics},
journal = {Inverse Problems},
year = {2020},
publisher = {{IOP} Publishing},
doi = {10.1088/1361-6420/ab8484},
volume = {36},
number = {7},
eid = {075012},
}

@Article{HL:25,
  author    = {Hohage, Thorsten and Liu, Meng},
  title     = {Passive Inverse Obstacle Scattering Problems for the {H}elmholtz Equation},
  doi       = {10.1137/24m1701812},
  issn      = {1095-712X},
  number    = {6},
  pages     = {2486–-2507},
  volume    = {85},
  journal   = {SIAM Journal on Applied Mathematics},
  month     = nov,
  publisher = {Society for Industrial & Applied Mathematics (SIAM)},
  year      = {2025},
}

@Book{ColtonKressIAaES,
author = {David Colton and Rainer Kress},
title = {Inverse Acoustic and Electromagnetic Scattering Theory},
publisher = {Springer Science+Business Media},
year = {2013},
edition = {3},
pages = {518},
doi = {10.1007/978-3-030-30351-8}
}

@article{HohageWeidling2015,
	year = 2015,
	publisher = {{IOP} Publishing},
	volume = {31},
	number = {7},
	eid = {075006},
	author = {Thorsten Hohage and Frederic Weidling},
	title = {Verification of a variational source condition for acoustic inverse medium scattering problems},
	journal = {Inverse Problems},
  doi = {10.1088/0266-5611/31/7/075006},
}

@article{Hofmann_2007,
  title = {A convergence rates result for {T}ikhonov regularization in {B}anach spaces with non-smooth operators},
  author = {B Hofmann and B Kaltenbacher and C Pöschl and O Scherzer},
  journal = {Inverse Problems},
  doi = {10.1088/0266-5611/23/3/009},
  number = {3},
  volume = {23},
  eid = {987},
  year = {2007},
  month = {4},
}

@article{hohage2017,
  title={Characterizations of variational source conditions, converse results, and maxisets of spectral regularization methods},
  author={Thorsten Hohage and Frederic Weidling},
  doi = {10.1137/16M1067445},
  journal={SIAM Journal on Numerical Analysis},
  volume={55},
  number={2},
  pages={598--620},
  year={2017},
  publisher={SIAM}
}

@article{mandache2001,
  title={Exponential instability in an inverse problem for the {S}chr{\"o}dinger equation},
  author={Niculae Mandache},
  doi = {10.1088/0266-5611/17/5/313},
  journal={Inverse Problems},
  volume={17},
  number={5},
  eid={1435},
  year={2001},
  publisher={IOP Publishing}
}

@article{dichristo2003,
  title={Examples of exponential instability for inverse inclusion and scattering problems},
  author={Michele Di Cristo and Luca Rondi},
  doi = {10.1088/0266-5611/19/3/313},
  journal={Inverse Problems},
  volume={19},
  number={3},
  eid={685},
  year={2003},
  publisher={IOP Publishing}
}

@article{Devaney1979,
author = {Anthony J Devaney},
title = {The inverse problem for random sources},
doi = {10.1063/1.524277},
journal = {Journal of Mathematical Physics},
volume = {20},
number = {8},
pages = {1687--1691},
year = {1979}
}

@article{PLi2017a,
  author = {Peijun Li and Ganghua Yuan},
  title = {Stability on the inverse random source scattering problem for the one-dimensional {H}elmholtz equation},
  doi = {10.1016/j.jmaa.2017.01.074},
  journal = {Journal of Mathematical Analysis and Applications},
  volume = {450},
  number = {2},
  pages = {872--887},
  year = {2017},
}

@article{YZhao2019,
  author = {Yue Zhao and Peijun Li},
  title = {Stability on the one-dimensional inverse source scattering problem in a two-layered medium},
  doi = {10.1080/00036811.2017.1399365},
  journal = {Applicable Analysis},
  volume = {98},
  number = {4},
  pages = {682--692},
  year  = {2019},
  publisher = {Taylor & Francis},
}

@article{JLi2020,
  author = {Jianliang Li and Tapio Helin and Peijun Li},
  title = {Inverse random source problems for time-harmonic acoustic and elastic waves},
  doi = {10.1080/03605302.2020.1774895},
  journal = {Communications in Partial Differential Equations},
  volume = {45},
  number = {10},
  pages = {1335--1380},
  year  = {2020},
  publisher = {Taylor & Francis},
}

@article{PLi2021c,
  author = {Li, Peijun and Wang, Xu},
  title = {Inverse Random Source Scattering for the {H}elmholtz Equation with Attenuation},
  doi = {10.1137/19M1309456},
  journal = {SIAM Journal on Applied Mathematics},
  volume = {81},
  number = {2},
  pages = {485--506},
  year = {2021},
}

@article{Cheng2016,
  title = {Increasing stability in the inverse source problem with many frequencies},
  author = {Jin Cheng and Victor Isakov and Shuai Lu},
  doi = {https://doi.org/10.1016/j.jde.2015.11.030},
  journal = {Journal of Differential Equations},
  volume = {260},
  number = {5},
  pages = {4786--4804},
  year = {2016},
}

@article{Lassas2008,
  title = {Inverse Scattering Problem for a Two Dimensional Random Potential},
  author = {Matti Lassas and Lassi Päivärinta and Eero Saksman},
  doi = {10.1007/s00220-008-0416-6},
  journal = {Communications in Mathematical Physics},
  volume = {279},
  number = {3},
  pages = {669--703},
  year = {2008},
}

@book{ReedSimon1980,
  title={I: Functional Ananlysis},
  author={Michael Reed and Barry Simon},
  series={Methods of Modern Mathematical Physics},
  year={1980},
  publisher={Academic Press, Inc. (London) Ltd.},
  pages = {416},
  edition = {1},
  volume = {1},
}

@misc{Faddeev1965,
  author = {L. D. Faddeev},
  title = {Increasing solutions of the {S}chr\"odinger equation},
  journal = {Dokl. Akad. Nauk SSSR},
  year = {1965},
  volume = {165},
  number = {3},
  pages = {514--517},
  url = {http://mi.mathnet.ru/dan31838}
}

@misc{Haehner1998,
  author = {Peter Hähner},
  title = {On Acoustic, Electromagnetic,
  and Elastic Scattering Problems
  in Inhomogeneous Media},
  year = {1998},
  pages = {275},
  note = {Habilitation Thesis},
  publisher={Univeristät Göttingen},
  url = {https://num.math.uni-goettingen.de/picap/pdf/E577.pdf}
}

@book{Bleistein2001,
  author = {N. Bleistein and J. W. Stockwell and J. K. Cohen},
  title = {Mathematics of Multidimensional Seismic Imaging, Migration, and Inversion},
  publisher = {Springer New York, NY},
  doi = {10.1007/978-1-4613-0001-4},
  edition = {1},
  pages = {210},
  year = {2001}
}

@article{Kaltenbacher2018,
  author = {Kaltenbacher, Manfred and Kaltenbacher, Barbara and Gombots, Stefan},
  title = {Inverse Scheme for Acoustic Source Localization using Microphone Measurements and Finite Element Simulations},
  doi = {10.3813/AAA.919204},
  journal = {Acta Acustica united with Acustica},
  year = {2018},
  month = {07},
  pages = {647--656},
  volume = {104},
}

@article{RausHaemarik2007,
  title = {On the quasioptimal regularization parameter choices for solving ill-posed problems},
  author = {T. Raus and U. Hämarik},
  doi = {10.1515/jiip.2007.023},
  journal = {Journal of Inverse and Ill-posed Problems},
  pages = {419--439},
  volume = {15},
  number = {4},
  year = {2007},
}

@article{HELIN2018132,
  title = {Correlation based passive imaging with a white noise source},
  author = {T. Helin and M. Lassas and L. Oksanen and T. Saksala},
  doi = {10.1016/j.matpur.2018.05.001},
  journal = {Journal de Mathématiques Pures et Appliquées},
  volume = {116},
  pages = {132--160},
  year = {2018},
}

@article{Alessandrini1988StableDO,
  title={Stable determination of conductivity by boundary measurements},
  author={Giovanni Alessandrini},
  doi = {10.1080/00036818808839730},
  journal={Applicable Analysis},
  year={1988},
  volume={27},
  number = {1--3},
  pages={153--172},
}

@article{li2023stability,
  author = {Li, Peijun and Liang, Ying},
  title = {Stability for Inverse Source Problems of the Stochastic {H}elmholtz Equation with a White Noise},
  journal = {SIAM Journal on Applied Mathematics},
  volume = {84},
  number = {2},
  pages = {687--709},
  year = {2024},
  doi = {10.1137/23M1586331},
}

@article{lindsey2000basic,
  title={Basic principles of solar acoustic holography--(invited review)},
  author={Lindsey, C and Braun, DC},
  doi = {10.1007/978-94-011-4377-6_17},
  journal={Solar Physics},
  volume={192},
  number={1},
  pages={261--284},
  year={2000},
  publisher={Springer}
}

@misc{Uhlmann2008,
  url = {http://www.jstor.org/stable/40233793},
  author = {Gunther Uhlmann and Jenn-Nan Wang},
  title = {Reconstructing Discontinuities Using Complex Geometrical Optics Solutions},
  journal = {SIAM Journal on Applied Mathematics},
  number = {4},
  pages = {1026--1044},
  publisher = {Society for Industrial and Applied Mathematics},
  volume = {68},
  year = {2008}
}

@article{Agaltsov_2020,
  author = {A D Agaltsov and T Hohage and R G Novikov},
  title = {Global uniqueness in a passive inverse problem of helioseismology},
  doi = {10.1088/1361-6420/ab77d9},
  journal = {Inverse Problems},
  year = {2020},
  month = {4},
  publisher = {IOP Publishing},
  volume = {36},
  number = {5},
  eid = {055004},
}

@book{garnierPassiveImagingAmbient2016,
	location = {Cambridge},
	title = {Passive Imaging with Ambient Noise},
	isbn = {978-1-107-13563-5},
	publisher = {Cambridge University Press},
	author = {Garnier, Josselin and Papanicolaou, George},
	year = {2016},
	doi = {10.1017/CBO9781316471807},
}

@article{mullerQuantitativePassiveImaging2024,
	title = {Quantitative passive imaging by iterative holography: the example of helioseismic holography},
	volume = {40},
	doi = {10.1088/1361-6420/ad2b9a},
	pages = {045016},
	number = {4},
	journal = {Inverse Problems},
	author = {Müller, Björn and Hohage, Thorsten and Fournier, Damien and Gizon, Laurent},
	year = {2024},
	note = {Publisher: {IOP} Publishing},
}

@article{chandler-wilde_wave-number-explicit_2008,
	title = {Wave-{Number}-{Explicit} {Bounds} in {Time}-{Harmonic} {Scattering}},
	volume = {39},
	issn = {0036-1410},
	url = {https://doi.org/10.1137/060662575},
	doi = {10.1137/060662575},
	number = {5},
	urldate = {2026-08-27},
	journal = {SIAM Journal on Mathematical Analysis},
	publisher = {Society for Industrial and Applied Mathematics},
	author = {Chandler-Wilde, Simon N. and Monk, Peter},
	month = jan,
	year = {2008},
	pages = {1428--1455},
}
